\documentclass[11pt]{amsart}

\usepackage[T1]{fontenc}
\usepackage[utf8]{inputenc}
\usepackage{textcomp}
\usepackage{lmodern}
\usepackage{microtype}

\usepackage{amssymb}
\usepackage{amsthm}
\usepackage{amscd}

\usepackage{hyperref}

\usepackage{mathscinet}
\usepackage[giveninits=true,doi=false,url=false,sortcites,backend=biber]{biblatex}
\usepackage{tikz}

\newtheorem{theorem}{Theorem}[section]

\newtheorem{lemma}[theorem]{Lemma}

\newtheorem{cor}[theorem]{Corollary}
\newtheorem{question}[theorem]{Question}

\theoremstyle{definition}
\newtheorem{definition}[theorem]{Definition}

\newtheorem{example}[theorem]{Example}

 \def \dom{\operatorname{dom}}
 
\allowdisplaybreaks[2]

\begin{document}

\title[What do Ultraproducts Remember?]{What do Ultraproducts Remember about the Original Structures?}
\author{Henry Towsner}
\date{\today}
\thanks{Partially supported by NSF grants DMS-1600263 and DMS-2054379}
\address {Department of Mathematics, University of Pennsylvania, 209 South 33rd Street, Philadelphia, PA 19104-6395, USA}
\email{htowsner@math.upenn.edu}
\urladdr{\url{http://www.math.upenn.edu/~htowsner}}

\begin{abstract}
We describe a syntactic method for taking proofs which use ultraproducts and translating them into direct, constructive proofs.
\end{abstract}

\maketitle

\section{Introduction}

Ultraproducts (and various related tools like graph limits and nonstandard analysis) have been a powerful tool in various areas \cite{MR2742582} of mathematics\footnote{Including functional analysis, especially Banach space theory \cite{MR552464,MR1342297} and operator algebras \cite{MR2006539,MR3694564}; commutative algebra \cite{MR2676525} and algebraic geometry \cite{MR2445111}; probability theory \cite{hoover:arrays, MR859372}; and combinatorics \cite{MR3583029,MR2964622,MR3012035,MR2950773,nonstandardmethodsramsey}.}.  When used in areas other than logic, an ultraproduct is usually used in the following way.
\begin{itemize}
\item We begin with a sequence of structures $\mathfrak{M}_1$, $\mathfrak{M}_2$, $\ldots$, $\mathfrak{M}_i$, $\ldots$, sometimes with the $\mathfrak{M}_i$ each finite in some sense (like cardinality, characteristic, or dimension) but with the finite value growing as $i$ does.
\item We construct an ultraproduct $\mathfrak{M}^{\mathcal{U}}$ and work with this object as in any other proof---successively proving $\mathfrak{M}^{\mathcal{U}}$ satisfies some lemma, then another lemma, and so on.
\item We finish by showing that $\mathfrak{M}^{\mathcal{U}}$ has some property which we can use to conclude something about the original structures $\mathfrak{M}_i$.  (In some cases a contradiction, thereby showing that the original sequence could not have existed.)
\end{itemize}

In such a proof, the essential properties of the ultraproduct must be determined by the original structures $\mathfrak{M}_i$.  While many specific questions about the ultraproduct might not be completely determined by the original structures (indeed, the point of the construction is that it uses an ultrafilter to resolve many questions about the ultraproduct in a coherent but arbitrary way), these questions must be inessential, because in actual applications the steps of the proof can be carried out without regard to how these additional questions are answered.

The intermediate steps of a proof using ultraproducts involve showing that the ultraproduct satisfies various properties.  In this paper we are concerned with the correspondence between these intermediate steps of a proof and the properties of the original structures.

\begin{quote}
  Suppose $\mathfrak{M}^{\mathcal{U}}$ has a property $P$.  What can we conclude about the original structures $\mathfrak{M}_i$?
\end{quote}

An answer to this question gives us a method for removing the use of ultraproducts from these proofs\footnote{Since nonstandard analysis can be formulated as an ultraproduct, this also means this method can be used to ``standardize'' nonstandard proofs.}: given a proof which uses ultraproducts, one simply replaces each step in the ultraproduct by the corresponding fact about the original structures.

As we will see, the general theme is that ultraproducts capture a certain kind of ``higher-order quantitative uniformity''.  Replacing ultraproduct arguments with standard ones requires making these uniformities explicit.  This explicitness comes at the cost of an argument which is typically much more complicated than the ultraproduct argument (as can be seen from examples, like \cite{MR3846327,MR3643744,1609.07509}, where this process has been applied).  However there are several reasons we might want these more explicit arguments anyway, for instance because they give explicit quantitative information, or because we want to see how the ultraproduct methods relate to ``classical'' methods.

When the ultraproduct $\mathfrak{M}^{\mathcal{U}}$ has a property $P$ which can be expressed in first-order logic, the corresponding property of the structures $\mathfrak{M}_i$ is given by the well-known \L o\'s Theorem---$P$ holds in $\mathfrak{M}^{\mathcal{U}}$ if and only $\{i\mid \mathfrak{M}_i\text{ has property }P\}$ belongs to the ultrafilter $\mathcal{U}$.

It follows that no substantive proof using ultraproducts will consider only properties which can be expressed in first-order logic: if it did, the ultraproduct would be entirely extraneous, since one could carry out the proof unchanged in most of the original structures.

So we are mostly interested in the case $P$ is \emph{not} a first-order property.  In most actual examples, like those cited above, the additional ingredient is quantification over the natural numbers.  (See Section \ref{sec:examples} for concrete examples of what this looks like.)  It turns out to be natural to consider a slightly more general setting: properties expressed in the logic $\mathrm{L}_{\omega_1,\omega}$, which allows arbitrary countably infinite conjunctions and disjunctions.

The sentences of $\mathrm{L}_{\omega_1,\omega}$ are often not interesting in the original structures $\mathfrak{M}_i$.  For instance, if the $\mathfrak{M}_i$ are finite sets of unbounded size, the ultraproduct $\mathfrak{M}^{\mathcal{U}}$ has the property ``this structure is infinite'', which is easily expressed in $\mathrm{L}_{\omega_1,\omega}$, but is false in all the $\mathfrak{M}_i$.

Our main project will be giving an alternate semantics for sentences of $\mathrm{L}_{\omega_1,\omega}$ which gives a meaningful interpretation in the original structures.  This will be a \emph{bounded} semantics, in which we replace infinite conjunctions and disjunctions with appropriate finite restrictions; it will then be sensible to ask when a sentence holds in an original structure in this bounded way.

The work of defining this semantics will occupy Section \ref{sec:bounded}, where we will ultimately define a notion
\[\mathfrak{M}\vDash^{\leq A,E}\sigma,\]
which will mean that the $\mathrm{L}_{\omega_1,\omega}$ sentence $\sigma$ holds in the structure $\mathfrak{M}$ \emph{subject to the bounds $A$ and $E$}.  Most of the difficulty will be establishing the right notion of what a bound is, which will become more complicated as $\sigma$ does.

We will then be able to establish our main result:
\begin{theorem}
  If $\{\mathfrak{M}_i\}_{i\in\mathbb{N}}$ is a sequence of structures, $\mathcal{U}$ is an ultrafilter, and $\sigma$ is a sentence in $\mathrm{L}_{\omega_1,\omega}$ then $\mathfrak{M}^{\mathcal{U}}\vDash\sigma$ if and only if, for every bound $A\in\mathbb{S}^\forall_\sigma$ there is a bound $E\in\mathbb{S}^\exists_\sigma$ such that
\[\{i\mid\mathfrak{M}_i\vDash^{\leq A,E}\sigma\}\in\mathcal{U}.\]
\end{theorem}

Keisler has already described a non-constructive method for relating the interpretation of an $L_{\omega_1,\omega}$ (as well as some sentences with infinitely many quantifiers) in an ultraproduct to the interpretation of related first-order sentences in the original models \cite{MR0202602}.

To obtain the constructive interpretation here, our bounded semantics is based heavily on Kohlenbach's monotone functional interpretation \cite{MR1428007}.  Similar interpretations have been given in the context of nonstandard analysis by Nelson \cite{MR938372} and by van den Berg, Briseid, and Safarik \cite{MR2964881}.  (See also \cite{MR3325571}, which gives an interpretation based on the bounded, rather than monotone, functional interpretation.) Most recent work applying these techniques has focused on questions about the logical strength of various principles in nonstandard settings, as in \cite{MR3647839,MR3682854,MR3599342}. 

Here our focus is on proofs from outside of logic: we can take a proof which uses an ultraproduct and translate the lemmas proved about that ultraproduct, one by one, into statements which are directly about the original structures.  This gives some insight into why ultraproducts  are useful in these proofs: the intermediate stages typically involve sentences $\sigma$ whose bounded form is rather complicated.  In this context it becomes clear that the power of the ultraproduct is to translate these complicated bounded statements into more conventional ones.  However---as Examples \ref{ex:convergence} and \ref{ex:regularity} will illustrate---the bounded form that emerges from the translation is sometimes interesting in its own right.  Furthermore, the bounded form typically has additional constructive\footnote{Our results in this paper might appear to contradict Sanders' claims in \cite{10.1007/978-3-319-72056-2_19}.  In that paper, Sanders discusses the ``metastability trade-off'', in which ``introducing a finite but arbitrarily large domain results in highly uniform and computable results''.  (See Example \ref{ex:convergence} for a longer discussion about the term metastability and how it relates to the construction in this paper.)  Sanders claims to ``place[] a hard limit on the generality of the metastability trade-off and show[] that metastability does not provide a 'king's road' towards computable mathematics.''  Yet that is precisely what we will do in this paper: by replacing quantification over infinite domains with the right arbitrarily large finite domains, we will obtain uniform and computable results, and we will do so in a quite general way.  The issue is an ambiguity about the correct way to generalize metastability to more complicated sentences: Sanders only considers one particular generalization, and so the limitations he finds only apply to that.  Here, as in the paper Sanders is responding to \cite{MR3278188} and elsewhere in the literature \cite{MR3643744}, we consider a different generalization based on the monotone functional interpretation.  Sanders' result shows that at least some of the complexity of this translation is necessary.} or explicit information that the ultraproduct form obscured.  We discuss examples of applications of this translation in Section \ref{sec:applications}.

\section{Notation}

When $\{\mathfrak{M}_i\}_{i\in\mathbb{N}}$ is a sequence of structures and $\mathcal{U}$ is a non-principal ultrafilter, we write $\mathfrak{M}^{\mathcal{U}}$ for the ultraproduct of this sequence.  See, for instance, \cite{MR2757537} for the details of the construction.  The crucial property is:
\begin{theorem}[\L{o}\'s Theorem]
  If $\sigma$ is a sentence of first-order logic then
\[\mathfrak{M}^{\mathcal{U}}\vDash\sigma\Leftrightarrow\{i\mid\mathfrak{M}_i\vDash\sigma\}\in\mathcal{U}.\]
\end{theorem}

Slightly more generally, if for each $i$ we have an element $a^i\in|\mathfrak{M}_i|$, we will write $\langle a^i\rangle_{\mathcal{U}}\in\mathfrak{M}^{\mathcal{U}}$ for the corresponding element of $\mathfrak{M}^{\mathcal{U}}$.  Then the \L{o}\'s Theorem also says:
\begin{theorem}
  If $\phi$ is a first-order formula with free variables $x_1,\ldots,x_n$ then
\[\mathfrak{M}^{\mathcal{U}}\vDash\phi[\langle a^i_1\rangle_{\mathcal{U}},\ldots,\langle a^i_n\rangle_{\mathcal{U}}]\Leftrightarrow\{i\mid\mathfrak{M}_i\vDash\phi[a^i_1,\ldots,a^i_n]\}\in\mathcal{U}.\]
\end{theorem}

In this paper we will only consider sequences indexed by $\mathbb{N}$.  The notions would carry over essentially unchanged to other countable index sets; similar ideas would apply with somewhat more care to $\omega$-incomplete ultrafilters on uncountable index sets.

We define formulas of $\mathrm{L}_{\omega_1,\omega}$ to be given by:
\begin{itemize}
\item every atomic formula (i.e. every formula $Rt_1\cdots t_n$ where $R$ is an $n$-ary predicate symbol and each $t_i$ is a term) is a formula of $\mathrm{L}_{\omega_1,\omega}$,
\item if $\phi$ is a formula of $\mathrm{L}_{\omega_1,\omega}$ then so is $\neg\phi$,
\item if $\phi$ is a formula of $\mathrm{L}_{\omega_1,\omega}$ and $x$ is a variable then $\forall x\phi$ is also a formula of $\mathrm{L}_{\omega_1,\omega}$,
\item if $I$ is either $\mathbb{N}$ or $\{0,1\}$, and $x_1,\ldots,x_k$ is a finite list of variables, for each $i\in I$, $\phi_i$ is a formula of $\mathrm{L}_{\omega_1,\omega}$ such that the only variables appearing free in $\phi_i$ are from the list $x_1,\ldots,x_k$, then $\bigwedge_{i\in I}\phi_i$ is a formula of $\mathrm{L}_{\omega_1,\omega}$.
\end{itemize}
  The clause about the variables in conjunctions and disjunctions ensures that even infinite formulas have only finitely many free variables. 

We do not include the existential quantifiers $\exists x\phi$ or $\bigvee_{i\in I}\phi_i$ in the formal definition; instead we will treat these as abbreviations for $\neg\forall x\neg\phi$ and $\neg\bigwedge_{i\in I}\neg\phi_i$ respectively.  Similarly, we treat $\phi\rightarrow\psi$ as an abbreviation for $\neg(\phi\wedge\neg\psi)$.

\begin{definition}
  We define the $\Pi^{\mathbb{N}}_n$ and $\Sigma^{\mathbb{N}}_n$ formulas by induction on $n$:
  \begin{itemize}
  \item any first-order formula is both $\Pi^{\mathbb{N}}_0$ and $\Sigma^{\mathbb{N}}_0$,
  \item if each $\phi_i$ is $\Sigma^{\mathbb{N}}_n$ then $\bigwedge_i\phi_i$ is $\Pi^{\mathbb{N}}_{n+1}$, and
\item if each $\phi_i$ is $\Pi^{\mathbb{N}}_n$ then $\bigvee_i\phi_i$ is $\Sigma^{\mathbb{N}}_{n+1}$.
  \end{itemize}
\end{definition}
This definition can be extended to sentences of infinite ordinal complexity, but we will not need to consider these in this paper.

\section{Motivating Examples}\label{sec:examples}

We begin by considering examples from several areas.  These examples are not new---they are already well-understood without the more complicated methods we introduce later---but serve to motivate the techniques introduced below. 

\begin{example}[Flatness of polynomial extensions]\label{ex:flatness}

For each $i$, let $k_i$ be a field, and let $k^{\mathcal{U}}$ be an ultraproduct of these fields; to avoid non-triviality, assume $k^{\mathcal{U}}$ is infinite (so either the $k_i$ are infinite, or at least their sizes are unbounded).

  There are two natural ways to think about rings of polynomials over $k^{\mathcal{U}}$: the usual ring of polynomials $k^{\mathcal{U}}[x]$, or the ultraproduct of the rings of polynomials $k_i[x]$, often denoted $k^{\mathcal{U}}_{int}[x]$.  The latter ring is larger, since in addition to the usual polynomials, it has additional elements like ``polynomials'' of nonstandard degree.  The ring $k^{\mathcal{U}}[x]$ is a subring of $k^{\mathcal{U}}_{int}[x]$, but is not itself an ultraproduct of any sequence of rings.  (In the language of nonstandard analysis, it is \emph{external}.)

  A key fact, often used when considering such rings (for instance, see \cite{MR2676525} for many examples), is
  \begin{quote}
    $k^{\mathcal{U}}_{int}[x]$ is flat over $k^{\mathcal{U}}[x]$.
  \end{quote}
A definition of flatness will be given shortly.  (We use a single variable only to simplify the notation; the same discussion would apply, essentially unchanged, with multiple variables.)

  This is a fact about ultraproducts of the sort we wish to consider.  If we are being completely formal, our language is the language of rings with an additional predicate $\mathbf{C}$ for denoting the constants and our original structures are the rings $k_i[x]$ with the predicate $\mathbf{C}$ holding exactly of $k_i\subseteq k_i[x]$.   $k^{\mathcal{U}}_{int}[x]$ is then the ultraproduct of these structures and $k^{\mathcal{U}}\subseteq k^{\mathcal{U}}_{int}[x]$ is the subset defined by the predicate $\mathbf{C}$.

  However flatness is not expressed by a first-order formula.  The substructure $k^{\mathcal{U}}[x]$ of $k^{\mathcal{U}}_{int}[x]$ cannot be defined using a first-order formula because it requires considering polynomials of arbitrary but finite degree---that is, it requires quantification over the natural numbers.  Therefore we need to use a sentence of $\mathrm{L}_{\omega_1,\omega}$ to write down a sentence which captures flatness.

There is a formulation of flatness in terms of solutions to linear equations; the fact above is equivalent to the $\Pi^{\mathbb{N}}_2$ sentence
  \begin{quote}
    For every homogeneous linear equation in finitely many variables with coefficients from $k^{\mathcal{U}}[x]$, any solution in $k^{\mathcal{U}}_{int}[x]$ is a $k^{\mathcal{U}}_{int}[x]$-linear combination of solutions in $k^{\mathcal{U}}[x]$.
  \end{quote}

Flatness can be expressed by a formula in the form
\[\bigwedge_{n\in\mathbb{N}}\bigwedge_{d\in\mathbb{N}}\forall \{c_{i,j}\}_{i\leq n,j\leq d}\subseteq \mathbf{C}\bigvee_{m\in\mathbb{N}}\bigvee_{b\in\mathbb{N}}\theta_{m,b}(\sum_{j\leq d}c_{1,j}x^j,\ldots,\sum_{j\leq d}c_{n,j}x^j)\]
where $\theta_{m,b}(t_1,\ldots,t_n)$ is a first-order formula which says
\begin{quote}
  any solution $\sum_{i\leq n}t_ig_i=0$ to the homogeneous linear equation $\sum_{i\leq n}t_iy_i=0$ can be written as a sum of at most $m$ solutions of degree at most $b$.
\end{quote}
(Here we are writing $\forall \{c_{i,j}\}_{i\leq n,j\leq d}\subseteq k^{\mathcal{U}}$ as an abbreviation for a long series of quantifiers $\forall c_{1,1}(\mathbf{C}c_{1,1}\rightarrow \forall c_{1,2}(\mathbf{C}c_{1,2}\rightarrow\cdots))$.)

A standard property of ultraproducts---countable saturation---lets us swap the first-order and countable quantifiers in this case, so the flatness of $k^{\mathcal{U}}_{int}[x]$ over $k^{\mathcal{U}}[x]$ is equivalent to
\[\bigwedge_{n\in\mathbb{N}}\bigwedge_{d\in\mathbb{D}}\bigvee_{m\in\mathbb{N}}\bigvee_{b\in\mathbb{N}}\forall \{c_{i,j}\}_{i\leq n,j\leq d}\subseteq \mathbf{C}\ \theta_{m,b}(\sum_{j\leq d}c_{1,j}x^j,\ldots,\sum_{j\leq d}c_{n,j}x^j).\]
The \L{o}\'s Theorem does not apply to this statement, but it does apply to the inner part, so, using this fact, we can derive a purely standard consequence.
\begin{theorem}
  For each $n$ and $d$ there are $m$ and $b$ so that whenever $k$ is a field and $p_1,\ldots,p_n\in k[x]$ are polynomials of degree at most $d$, any solution to the equation $\sum_i p_iy_i=0$ is a linear combination of at most $m$ solutions consisting of polynomials of degree at most $b$.
\end{theorem}
\begin{proof}
  Suppose not.  Then, for some $n,d$ it must be the case that for every $m,b$ there is a field $k_{m,b}$, polynomials $p_1,\ldots,p_n\in k_{m,b}[x]$ of degree at most $d$, and solutions $g_1,\ldots,g_n\in k_{m,b}[x]$ with $\sum_i p_ig_i=0$ such that there do not exist polynomials $\{h_{i,i'}\}_{i\leq n, i'\leq m}\subseteq k_{m,b}[x]$ of degree at most $b$ such that both for each $i'\leq m$, $\sum_i p_ih_{i,i'}=0$, and there exist $f_1,\ldots,f_{i'}\in k_{m,b}[x]$ so that $g_i=\sum_{i'}h_{i,i'}$ for each $i\leq n$.

  We let $k_m=k_{m,m}$ and consider any ultraproduct $k^{\mathcal{U}}$.  By assumption, there must be some $m, b$ such that
\[k^{\mathcal{U}}\vDash \forall \{c_{i,j}\}_{i\leq n,j\leq d}\subseteq \mathbf{C}\ \theta_{m,b}(\sum_{j\leq d}c_{1,j}x^j,\ldots,\sum_{j\leq d}c_{n,j}x^j),\]
 and by the \L{o}\'s Theorem, for almost every $m'$ we have 
\[k_{m',m'}\vDash \forall \{c_{i,j}\}_{i\leq n,j\leq d}\subseteq \mathbf{C}\ \theta_{m,b}(\sum_{j\leq d}c_{1,j}x^j,\ldots,\sum_{j\leq d}c_{n,j}x^j).\]
 In particular, we can find such an $m'$ with $m'\geq\max\{m,b\}$, which contradicts our assumption: since 
\[k_{m',m'}\vDash \forall \{c_{i,j}\}_{i\leq n,j\leq d}\subseteq \mathbf{C}\ \theta_{m,b}(\sum_{j\leq d}c_{1,j}x^j,\ldots,\sum_{j\leq d}c_{n,j}x^j),\]
 in particular $k_{m',m'}\vDash\theta_{m,b}(p_1,\ldots,p_n)$, so any solution $g_1,\ldots,g_n\in k_{m',m'}[x]$ \emph{is} a linear combination of at most $m$ solutions of degree at most $b$.
\end{proof}

\end{example}

This sort of equivalence is what we expect of properties expressed by $\Pi^{\mathbb{N}}_2$ formulas---we will have $\mathfrak{M}^{\mathcal{U}}\vDash \bigwedge_n\bigvee_m\phi_{n,m}$ exactly when, for each $n$ there is an $m$ such that for almost every $i$, $\mathfrak{M}_i\vDash\phi_{n,m}$.  (This is basically the content of the \emph{transfer theorem} often used in nonstandard analysis.)  That is, truth of the statement in the ultraproduct is equivalent to a certain kind of uniform truth in the original structures: we can find a bound $n$ which depends on $m$, but not on the particular structure.

\begin{example}[Topological Recurrence]
Recall van der Waerden's Theorem,
\begin{theorem}
  For any positive integers $k, r$, there is an $n$ so that whenever $[0,n]=U_0\cup\cdots\cup U_r$ is a partition of $[0,n]$, there is an $i\leq r$, an $a$, and an $s$ so that $a,a+s,a+2s,\ldots,a+(k-1)s\in U_i$.
\end{theorem}

Furstenberg and Weiss gave a topological proof \cite{MR531271} using the following theorem as a key step:
\begin{theorem}
  Let $X$ be a non-empty compact topological space and let $T:X\rightarrow X$ be a homeomorphism.  Then for any open cover $X=U_0\cup\cdots\cup U_r$ of $X$ and any $k\geq 1$, there exists $i\leq r$, an $x\in X$, and an integer $s$ so that $x$, $T^sx$, $T^{2s}x$, $\ldots$, $T^{(k-1)s}x\in U_i$.
\end{theorem}

Formally, we work in a language with a unary function symbol $\mathbf{T}$ and $r+1$ predicate symbols, which we may as well denote $\mathbf{U}_0$, $\ldots$, $\mathbf{U}_r$.  Suppose van der Waerden's Theorem were false---for some $k, r$ there exists an $n$ and a partition $[0,n]=U_0\cup\cdots\cup U_r$ so that no $U_i$ contains an arithmetic progression of length $k$.

We view each of these examples as a finite structure $([0,n],S,U_0,\ldots,U_r)$ where $S$ is the function mapping $a$ to $a+1\mod n$.  An ultraproduct of such structures gives rise to a topological space $X$ with a partition $X=U_0\cup\cdots\cup U_r$ in which the theorem above implies
\begin{quote}
  For every $k\in\mathbb{N}$ there is some $i\leq r$, some $x\in X$, and some $s\in\mathbb{N}$ so that $x$, $T^sx$, $T^{2s}x$, $\ldots$, $T^{(k-1)s}x\in U_i$.
\end{quote}
Since $k$ and $s$ are natural numbers, this is not a statement of first-order logic, but this is a statement in $\mathrm{L}_{\omega_1,\omega}$:
\[\bigwedge_{k\in\mathbb{N}}\bigvee_{i\leq r}\exists x\bigvee_{s\in\mathbb{N}} (x\in \mathbf{U}_i)\wedge (\mathbf{T}^sx\in \mathbf{U}_i)\wedge\cdots\wedge (\mathbf{T}^{(k-1)s}x\in \mathbf{U}_i).\]
In particular, there must be some $i\leq r$, some $x\in X$, and some $s\in\mathbb{N}$ so that $T^{js}\in U_i$ for all $j\in[0,2k-1]$.  Choosing $n\geq (2k-1)s$, the \L{o}\'s Theorem implies that there is an $x\in[0,n]$ so that $x+js\mod n\in U_i$ for all $j\in[0,2k-1]$.  In particular, either $x,x+s,\ldots,x+(k-1)s\in U_i$ or (if $x+(k-1)s>n$), $x+(k-1)s\mod n,x+ks\mod n,\ldots,x+(2k-1)s\mod n\in U_i$.
\end{example}

In this example, the property we were concerned with was again expressed by a $\Pi^{\mathbb{N}}_2$ formula.

\begin{example}[Convergence]\label{ex:convergence}
  Examples involving convergence of sequences occur both in applications involving ergodic theory (as in \cite{tao08Norm}) and functional analysis (as in \cite{towsner:banach}).  To avoid issues of measurability, consider a pure metric space, so the language includes a distance function.  More precisely, to stay in the realm of first-order languages, suppose we have a language containing countably many binary predicates $\mathbf{d}_{<1/n}$ and countably many constant symbols $\mathbf{c}_k$.

  Suppose we have a collection of metric spaces $(X_i,d_i)$ of diameter $1$ and, in each of these spaces, a sequence $\{a^i_k\}$.  We can interpret these metric spaces as structures in our language: $|\mathfrak{M}_i|=X_i$, $\mathbf{d}^{\mathfrak{M}_i}_{<1/n}(x,y)$ holds exactly when $d_i(x,y)<1/n$, and $\mathbf{c}_k^{\mathfrak{M}_i}=a^i_k$.

  The ultraproduct $\mathfrak{M}^{\mathcal{U}}$ gives rise to a pseudo-metric space where we define $d(x,y)=\sup\{1/n\mid \mathbf{d}_{<1/n}^{\mathfrak{M}^{\mathcal{U}}}(x,y)\}$.  (We could obtain a proper metric space by taking a quotient by the equivalence relation $x\sim y$ iff $d(x,y)=0$.)  We have a sequence $a_k=\mathbf{c}^{\mathfrak{M}^{\mathcal{U}}}_k$ in this metric space.  (That is, $a_k=\langle a^i_k\rangle_{\mathcal{U}}$.)

  The statement that the sequence $a_k$ converges is not first-order, but can once again be expressed in $\mathrm{L}_{\omega_1,\omega}$:
\[\bigwedge_{n\in\mathbb{N}}\bigvee_{k\in\mathbb{N}}\bigwedge_{k'\in\mathbb{N}} k'\geq k\rightarrow \mathbf{d}_{<1/n}(\mathbf{c}_k,\mathbf{c}_{k'}).\]

This is $\Pi^{\mathbb{N}}_3$, not $\Pi^{\mathbb{N}}_2$, and the behavior is more subtle.  Unlike the previous examples, convergence of $\langle a_k\rangle$ does not imply that most of the sequences $\langle a^i_k\rangle$ converge.  For instance, if each $X_i=\{0,1\}$ with $d(0,1)=1$ and the sequences are given by $a^i_k=\left\{\begin{array}{ll}1&\text{if }k<i\text{ or }k\text{ is even}\\0&\text{otherwise}\end{array}\right.$ then the sequence $\langle a^i_k\rangle$ do not converge for any $i$, but the ultraproduct $X=\{0,1\}$ and $a_k$ is constantly equal to $1$, so the sequence $\langle a_k\rangle$ converges trivially.

Instead, convergence of $\langle a_k\rangle$ is equivalent to a more complicated property---the ``uniform metastable convergence''---of the $\langle a^i_k\rangle$ (\cite{tao08Norm,MR2550151, MR2144170, MR2130066}).  The relationship between metastable convergence and ultraproducts was studied explicitly in \cite{MR3141811,iovino_duenez}.

\begin{theorem}
  The sequence given by $a_k=\mathbf{c}^{\mathfrak{M}^{\mathcal{U}}}_k$ converges in $\mathfrak{M}^{\mathcal{U}}$ if and only if for every $\epsilon>0$ and every $F:\mathbb{N}\rightarrow\mathbb{N}$ there is an $m\in\mathbb{N}$ such that for $\mathcal{U}$-almost every $i$, $d_i(a^i_m,a^i_{\max\{m,F(m)\}})<\epsilon$.
\end{theorem}
\begin{proof}
  Suppose the second statement is false, so there is an $\epsilon>0$ and an $F$ such that, for every $m\in\mathbb{N}$, $\{i\mid d_i(a^i_m,a^i_{\max\{m,F(m)\}})\geq\epsilon\}$.  Replacing $F$ with $F'(m)=\max\{m,F(m)\}$, we may assume $m\leq F(m)$ for all $m$.

  Pick $n$ with $1/n<\epsilon$.  Then, for every $m$, $\{i\mid \mathfrak{M}_i\vDash \neg \mathbf{d}_{<1/n}(\mathbf{c}_m,\mathbf{c}_{F(m)})\}\in\mathcal{U}$, so the \L{o}\'s Theorem implies that $\mathfrak{M}^{\mathcal{U}}\vDash \neg \mathbf{d}_{<1/n}(\mathbf{c}_m,\mathbf{c}_{F(m)})\}$, and therefore, for every $m$, $d(a_m,a_{F(m)})\geq 1/n$.  Clearly $F(m)>m$ (since $d(a_m,a_m)=0$), so for every $m$ there is an $m'>m$ with $d(a_m,a_{m'})\geq 1/n$, so the sequence $a_k$ does not converge.

Conversely, suppose the sequence $\langle a_k\rangle$ does not converge, so for some $\epsilon>0$ and every $m$, there is an $m'>m$ with $d(a_m,a_{m'})>\epsilon$.  Choose $n$ with $1/n<\epsilon$ and, for each $m$, define $F(m)$ to be the least $m'>m$ such that $d(a_m,a_{m'})>1/n$.  Then, for each $m$, $\mathfrak{M}^{\mathcal{U}}\vDash \neg \mathbf{d}_{<1/n}(\mathbf{c}_m,\mathbf{c}_{F(m)})$, so $\{i\mid \mathfrak{M}^{\mathcal{U}}\vDash \neg \mathbf{d}_{<1/n}(\mathbf{c}_m,\mathbf{c}_{F(m)})\}\in\mathcal{U}$, so $\{i\mid d_i(a^i_m,a^i_{\max\{m,F(m)\}})\}\in\mathcal{U}$.
\end{proof}
\end{example}

\section{A Bounded Semantics}\label{sec:bounded}

\subsection{Uniformity}

Our main goal is define a ``bounded semantics'' $\mathfrak{M}\vDash^{\leq A,E}\phi[\vec b]$ when $\phi$ is a formula of $\mathrm{L}_{\omega_1,\omega}$.  Our main theorem (specailized to sentences for notational simplicity) will say
\begin{quote}
When $\phi$ is a sentence of $\mathrm{L}_{\omega_1,\omega}$, $\mathfrak{M}^{\mathcal{U}}\vDash\phi$ if and only if, for each $A$, there is an $E$ so that $\{i\mid\mathfrak{M}_i\vDash^{\leq A,E}\phi\}\in\mathcal{U}$.
\end{quote}
That is, $\phi$ is true in the ultraproduct when $\phi$ is ``uniformly'' true in the structures $\mathfrak{M}_i$---when, for each $A$, there is a single $E$ witnessing the truth of $\phi$ which works for almost all $i$.

We can think of this semantics as coming from a game: given a sentence $\phi$ of first-order logic and a structure $\mathfrak{M}$, we have two players, \textsc{Prover} and \textsc{Refuter}, playing a game to determine whether $\phi$ is true.

The first motivating case is when $\phi$ is $\Pi_2^{\mathbb{N}}$---that is, $\phi$ has the form $\bigwedge_{n\in\mathbb{N}}\bigvee_{m\in\mathbb{N}}\phi_{n,m}$ where each $\phi_{n,m}$ is a first-order sentence.  In this case, we will define
\[\mathfrak{M}\vDash^{\leq n,m}\bigwedge_{n\in\mathbb{N}}\bigvee_{m\in\mathbb{N}}\phi_{n,m}\text{ if and only if }\mathfrak{M}\vDash\phi_{n,m}.\]
In the corresponding game, \textsc{Refuter} first chooses an $n$, and then \textsc{Prover} chooses a $m$; \textsc{Prover} wins if $\phi_{n,m}$ is true in $\mathfrak{M}$.  That means that \textsc{Refuter} has a winning strategy if there is some $n$ so that, for all $m$, $\phi_{n,m}$ is false in $\mathfrak{M}$, while \textsc{Prover} has a winning strategy if, for every $n$, there is some $m$ so that $\phi_{n,m}$ is true in $\mathfrak{M}$.

So far this is just a restatement of the usual meaning of the quantifiers.  Note that we are treating truth of first-order sentences (even complex ones) as ``atomic''---once the moves break $\phi$ down to a first-order component $\phi_{n,m}$, we simply inquire whether it is true.\footnote{The formal inductive definition will be superficially more complicated in this case, but it will still be true that the substantive steps take place only in formulas with countable conjections or disjunctions.}

For a $\Pi_3^{\mathbb{N}}$ sentence $\bigwedge_{n\in\mathbb{N}}\bigvee_{m\in\mathbb{N}}\bigwedge_{k\in\mathbb{N}}\phi_{n,m,k}$, one might expect the game where first \textsc{Refuter} chooses an $n$, then \textsc{Prover} chooses a $m$, and then \textsc{Refuter} chooses a $k$.  This turns out not to be the right interpretation for our purpose: we need to make the strategies more explicit.  In the context of a $\Pi_3^{\mathbb{N}}$ sentence, that means this first attempt at a game is unfair to \textsc{Prover}: \textsc{Refuter} is allowed to look at $m$ before choosing $k$, but \textsc{Prover} isn't being allowed to look at \textsc{Refuter}'s \emph{strategy} for choosing $k$ before choosing $m$.

The right game for $\bigwedge_{n\in\mathbb{N}}\bigvee_{m\in\mathbb{N}}\bigwedge_{k\in\mathbb{N}}\phi_{n,m,k}$, for our purposes, is the game in which \textsc{Refuter} first chooses \emph{both} an $n$ and a ``strategy''.  \textsc{Prover} is allowed to look at, not only the witness $m$, but the function $K$, and is only required to choose a $m$ so that $\phi_{n,m,K(m)}$ is true.\footnote{Actually, in order to make the inductive step work, it turns out to be cleaner to ask $K(m)$ to be a \emph{bound} rather than an exact value---we ultimately ask \textsc{Prover} to choose a $m$ so that $\phi_{n,m,k}$ is true for every $k\leq K(m)$.}

This corresponds to saying
\[\mathfrak{M}\vDash^{(n,K),m}\bigwedge_{n\in\mathbb{N}}\bigvee_{m\in\mathbb{N}}\bigwedge_{k\in\mathbb{N}}\phi_{n,m,k}\text{ if and only if }\mathfrak{M}\vDash\phi_{n,m,K(m)}.\]

In fact,  this game is equivalent to the first attempt: certainly, if \textsc{Prover} has a winning strategy in the first game, this same strategy works in the second game---\textsc{Prover} ignores $K$, plays the winning $m$, and would win against any value of $k$, so certainly wins against the play $K(j)$.  Conversely, if \textsc{Refuter} has a winning strategy in the first game, that means there is an $n$ so that, for every $m$, there is a $k$ which is a winning move for \textsc{Refuter} (that is, so that $\phi_{n,m,k}$ is false); then the function $K$ which, given $m$, outputs $k$, is a winning strategy for the second version of the game.

We are ultimately interested, not in when \textsc{Prover} can win in a single structure, but when \textsc{Prover} can win \emph{uniformly}: \textsc{Prover} is supposed to produce a single $m$ depending only on \textsc{Refuter}'s moves, which works in many different structures $\mathfrak{M}_i$.  This makes the game harder for \textsc{Prover}, so we balance that by allowing \textsc{Prover} to get more information---the choice of $m$ can depend on $K$.

In general, we will define our notion of strategies by induction on the formula $\phi$.  For each $\phi$, we will have two sets---$\mathbb{S}^\forall_\phi$ and $\mathbb{S}^\exists_\phi$---representing \textsc{Refuter}'s and \textsc{Prover}'s strategies, respectively.

For instance, suppose $\phi$ is $\psi_0\wedge\psi_1$, and we have already identified the sets $\mathbb{S}^\forall_{\psi_b}$ and $\mathbb{S}^\exists_{\psi_b}$.  Then we would expect $\mathbb{S}^\forall_{\phi}$ to be something like
\[(\{0\}\times\mathbb{S}^\forall_{\psi_0})\cup(\{1\}\times\mathbb{S}^\exists_{\psi_1}).\]
That is, \textsc{Refuter}'s move is to specify \emph{both} a choice $b\in\{0,1\}$ and a strategy for the chosen $\psi_b$.  Similarly, $\mathbb{S}^\exists_\phi$ would be something like $\mathbb{S}^\exists_{\psi_0}\cup\mathbb{S}^\exists_{\psi_1}$---\textsc{Prover} must respond to whichever $\psi_b$ \textsc{Refuter} has chosen with an appropriate strategy.

The case which leads to most of the complication is negation.  (Recall that $\bigvee_{m\in\mathbb{N}}$ is an abbreviation for $\neg\bigwedge_{m\in\mathbb{N}}\neg$, so this case already emerged, implicitly, in the examples above.)  Suppose $\phi$ is $\neg\psi$, and we have already identified the sets $\mathbb{S}^\forall_\psi$ and $\mathbb{S}^\exists_\psi$.  We need \textsc{Refuter} and \textsc{Prover} to switch roles; a first attempt would be to simply swap them: in the game for $\phi$, \textsc{Prover} first plays a strategy from $\mathbb{S}^\forall_\psi$, and then \textsc{Refuter} responds with a strategy from $\mathbb{S}^\exists_\psi$.

Once again, it is too difficult for \textsc{Prover} to win this game uniformly, so we need to give \textsc{Prover} more information.  So, before \textsc{Prover} chooses a strategy, \textsc{Refuter} has to declare a strategy: \textsc{Refuter} has to play a \emph{function} $F$ from $\mathbb{S}^\forall_\psi$ to $\mathbb{S}^\exists_\psi$, and \textsc{Prover} is then required to play an element $A$ of $\mathbb{S}^\forall_\psi$ with the property that, in the game corresponding to $\psi$, $A$ defeats the particular strategy $F(A)$. Moreover, \textsc{Refuter}'s strategy needs to be continuous in a suitable sense---approximations to $F(A)$ should be determined by approximations to $A$.

Consider the simplest case where this comes up, where $\phi$ has the form $\neg\bigvee_{n\in\mathbb{N}}\bigwedge_{m\in\mathbb{N}}\phi_{n,m}$.  Let us write $\psi$ for the sentence $\bigvee_{n\in\mathbb{N}}\bigwedge_{m\in\mathbb{N}}\phi_{n,m}$, so $\phi$ is $\neg\psi$.  This is slightly simpler than $\Pi^{\mathbb{N}}_3$ case above: \textsc{Refuter}'s strategy for $\psi$ is to play a function $M$, and \textsc{Prover}'s strategy for $\psi$ is to play a number $n$; \textsc{Prover} wins if $\phi_{n,M(n)}$ is true.  That is, $\mathbb{S}^\forall_{\psi}=\mathbb{N}^{\mathbb{N}}$ and $\mathbb{S}^\exists_\psi=\mathbb{N}$.

So a \textsc{Refuter} strategy for $\phi$ should be a ``functional'': a function $\mathcal{F}:\mathbb{N}^{\mathbb{N}}\rightarrow\mathbb{N}$.  \textsc{Prover} responds with an element of $\mathbb{S}^\forall_\psi$---a function $M$---and then \textsc{Prover} wins if $\phi_{F(M),M(F(M))}$ is false.

We need to break down the way that \textsc{Refuter}'s strategy for $\phi$ should be continuous; let us follow this back to thinking about what a winning \textsc{Prover} strategy for $\psi$ looks like.  A winning \textsc{Prover} strategy for $\psi$ responds to a function $M$ with a number $n$; having identified the winning response $n$, it doesn't matter what $M$ does on other input---if $M'$ is a different function with $M'(n)=M(n)$, $n$ is a winning play against $M'$, as well.

We might imagine that \textsc{Prover} finds the winning response $n$ by simulating some interaction with $M$: first \textsc{Prover} tries $n_0$ and, if $\phi_{n_0,M(n_0)}$ is false, tries a second value $n_1$, and so on, until finding the winning play $n=n_k$.  (This is, in fact, what winning strategies for \textsc{Prover} will look like.)  Any $M'$ which agrees with $M$ on the finitely many values $n_0,\ldots,n_k$ will lead to \textsc{Prover} identifying the same winning response $n$.

So when we apply a negation and consider \textsc{Refuter}'s strategy for $\phi$, the function $F$ should have the same property: $F(M)$ depends only on $M(n)$ for finitely many values of $n$.  That is:
\begin{quote}
 whenever $F(M)=n$, there are finitely many values $n_0,\ldots,n_k$ such that whenever $M'(n_i)=M(n_i)$ for all $i\leq k$, $F(M')=n$.
\end{quote}
This is precisely saying that $F$ is a continuous function if the topology on $\mathbb{N}$ is discrete and the topology on $\mathbb{N}^{\mathbb{N}}$ is the one basic open sets have the form $\{M\mid M(n_0)=m_0,\ldots,M(n_k)=m_k\}$ for finitely many pairs $(n_0,m_0),\ldots,(n_k,m_k)$.

\subsection{Continuity}

The ideas in the previous subsection are implemented in Definition \ref{def:strategies}, where we will define, for each formula $\mathrm{L}_{\omega_1,\omega}$ formula $\phi$, two sets $\mathbb{S}^\forall_\phi$ and $\mathbb{S}^\exists_\phi$ such that:
\begin{itemize}
\item when $\phi$ is atomic, $\mathbb{S}^\forall_\phi$ and $\mathbb{S}^\exists_\phi$ are sets with a single point $\{\ast\}$,
\item when $\phi$ is $\forall x\,\psi$, $\mathbb{S}^\forall_\phi=\mathbb{S}^\forall_\psi$ and $\mathbb{S}^\exists_\phi=\mathbb{S}^\exists_\psi$,
\item when $\phi$ is $\psi_0\wedge\psi_1$,
  \begin{itemize}
  \item $\mathbb{S}^\forall_\phi=(\{0\}\times\mathbb{S}^\forall_{\psi_0})\cup(\{1\}\times\mathbb{S}^\exists_{\psi_1})$,
  \item $\mathbb{S}^\exists_\phi=\mathbb{S}^\exists_{\psi_0}\cup\mathbb{S}^\exists_{\psi_1}$,
  \end{itemize}
\item when $\phi$ is $\bigwedge_{i\in\mathbb{N}}\psi_i$ then:
  \begin{itemize}
  \item $\mathbb{S}^\forall_\phi\subseteq\bigcup_{\emptyset\subsetneq J\subseteq_{\mathrm{fin}}I}\prod_{j\in J}\mathbb{S}^\forall_{\psi_j}$,
  \item $\mathbb{S}^\exists_\phi=\bigcup_{J\subseteq_{\mathrm{fin}}I}\prod_{j\in J}\mathbb{S}^\exists_{\psi_j}$,
  \end{itemize}
\item when $\phi$ is $\neg\psi$,
  \begin{itemize}
  \item $\mathbb{S}^\forall_\phi\subseteq(\mathbb{S}^\exists_\psi)^{\mathbb{S}^\forall_\psi}$,
  \item $\mathbb{S}^\exists_\phi=\mathbb{S}^\forall_\psi$.
  \end{itemize}
\end{itemize}

When $\phi$ is $\bigwedge_{i\in \mathbb{N}}\psi_i$, \textsc{Refuter}'s strategies consist of a finite set $J$ and, for each $j\in J$, a strategy for refuting $\mathbb{S}^\forall_{\psi_j}$.  This will turn out to mean that \textsc{Refuter} plans to refute $\psi_j$ for \emph{some} $j\in J$, but may not have decided which yet---indeed, it could be that the specfic choice of $j$ for which \textsc{Refuter} wins depends on what strategy \textsc{Prover} chooses.  Therefore \textsc{Prover} needs to respond with a winning strategy for every $j\in J$.\footnote{The exact definition of the $\bigwedge$ cases used below is not quite the same.  First, the definitions are tweaked so that they can be treated as a single case, and second, $\mathbb{S}^\exists_{\bigwedge_i \psi_i}$ is modified so that \textsc{Prover} has the option of choosing a single strategy for all $i$ simultaneously, which has the effect of simplifying the structure of the witnesses in certain cases.}

In order to define continuity, we will need a suitable topology on the spaces $\mathbb{S}^Q_\phi$: we will define a set of ``partial strategies'' $\mathbb{P}^Q_\phi$, and take $\mathbb{S}^Q_\phi$ to be the limit points of $\mathbb{P}^Q_\phi$.

Consider the representative case where $\mathbb{S}^Q_\phi$ consists of functions from $\mathbb{N}$ to $\mathbb{N}$; then elements of $\mathbb{P}^Q_\phi$ will be finite partial functions from $\mathbb{N}$ to $\mathbb{N}$.  As is conventional, when $f,g\in\mathbb{P}^Q_\phi$, we write $f\subseteq g$ when $g$ extends $f$, and when $F\in\mathbb{S}^Q_\phi$ we can write $f\subseteq F$ for the same.

Motivated by this, for every $\phi$ and $Q$ we will have an \emph{information ordering}, $\subseteq$, on the sets $\mathbb{P}^Q_\phi$: when $f,g\in\mathbb{P}^Q_\phi$, we write $f\subseteq g$ to mean ``$g$ is consistent with and gives more information than $f$''.  For instance, if we are considering partial functions $f,g$ approximating $(\mathbb{N}^{\mathbb{N}})^{\mathbb{N}}$, we have $f\subseteq g$ if:
\begin{itemize}
\item $\dom(f)\subseteq\dom(g)$, and
\item whenever $a\in\dom(f)$, $f(a)\subseteq g(a)$.
\end{itemize}
That is, $g$ can extend $f$ by extending the domain of definition, but also by extending the output. (As the notion suggests, it's generally true that $f\subseteq g$ also means that, encoded as sets, $f$ is literally a subset of $g$.)

In Definition \ref{def_P_sets}, we define the collections of partial strategies $\mathbb{P}_\phi^Q$ inductively.  We throw away certain elements which are internally inconsistent, and in Definition \ref{def_F_sets} we define $\mathbb{F}^Q_\phi\subseteq\mathbb{P}^Q_\phi$ to be those partial strategies which are internally consistent.

For example, suppose $\mathbb{P}^Q_\phi$ consists of partial functions approximating $\mathbb{N}^{\mathbb{N}^{\mathbb{N}}}$; then some $f\in\mathbb{P}^Q_\phi$ might have the property that whenever $g(1)=3$, $f(g)=3$ as well. Some $f'$ 

We will say that $f,g\in\mathbb{P}^Q_\phi$ are ``coherent'' if they have a common extension---if there is an $h\in\mathbb{F}^Q_\phi$ so that $f\subseteq h$ and $g\subseteq h$.  The elements of $\mathbb{S}^Q_\phi$ will then correspond to maximal coherent sets: whenever $\mathcal{C}\subseteq\mathbb{F}^Q_\phi$ is a maximal coherent set, there will be an $F=\bigcup\mathcal{C}\in\mathbb{S}^Q_\phi$.  When $\mathbb{S}^Q_\phi=\mathbb{N}^{\mathbb{N}}$, this works in the way we might expect: given a chain of partial functions $f_1\subseteq f_2\subseteq f_3\subseteq\cdots$ so that $\bigcup_n\dom(f_n)=\mathbb{N}$, the union $\bigcup_n f_n$ is a total function belonging to $\mathbb{S}^Q_\phi$.  More generally, $f_n(a)$ might itself be a partial function, so we need to ensure not only that the union of the domains is the entire domain, but that each $\bigcup_n f_n(a)$ is itself total.

At this point we can define our topology: for each $f\in\mathbb{F}^Q_\phi$, we have the open ball $B_f\subseteq\mathbb{S}^Q_\phi$ consisting of those $F$ such that $f\subseteq F$.  Our definitions will ensure that $\mathbb{S}^\forall_{\neg\psi}$ contains only continuous functions from $\mathbb{S}^\forall_\psi$ to $\mathbb{S}^\exists_\psi$.

We also need the definition of $\mathfrak{M}\vDash^{A,E}\phi$ to be continuous.  That is, we should have the property:
\begin{quote}
  given $\mathfrak{M}$, for every $(A,E)$ there should be an $a\subseteq A$ and an $e\subseteq E$ such that:
  \begin{itemize}
  \item if $\mathfrak{M}\vDash^{\leq A,E}\phi$ then whenever $(A',E')\in B_a\times B_e$, $\mathfrak{M}\vDash^{\leq A',E'}\phi$,
  \item if $\mathfrak{M}\not\vDash^{\leq A,E}\phi$ then whenever $(A',E')\in B_a\times B_e$, $\mathfrak{M}\not\vDash^{\leq A',E'}\phi$.
  \end{itemize}
\end{quote}
This will, in effect, be shown in Lemma \ref{thm:fragment}.  Given the total strategies $A$ and $E$, we can imagine playing out the game between \textsc{Prover} and \textsc{Refuter}; because the strategies are continuous, the outcome of the game will depend on only a finite portion of the strategies, so we can take $a$ and $e$ to be large enough fragments to determine the outcome.

\subsection{Finite Fragments}

We first define the sets $\mathbb{P}^\forall_\phi$ and $\mathbb{P}^\exists_\phi$ for each formula $\phi$.  These sets represent our basic building blocks: for $Q\in\{\forall,\exists\}$, the elements of $\mathbb{P}^Q_\phi$ are ``finite fragments'' of data about a bound on $\phi$.  Since most of our bounds will ultimately be functions of various kinds, the elements of $\mathbb{P}^Q_\phi$ are mostly partial functions.

  We will have two relations on these sets.  We will say $f\subseteq g$ if $g$ is a larger finite fragment than $f$ (that is, if $g$ agrees with $f$ and also might provide additional information); for instance, when $f$ and $g$ are partial functions, $f\subseteq g$ will hold when $g$ is literally an extension of $f$.  We will say $f\leq g$ if $g$ provides larger bounds; when $f$ and $g$ are natural numbers, this will be the usual ordering, and when $f$ and $g$ are partial functions this will generally mean that $f(i)\leq g(i)$ for all $i$ in some set.

\begin{definition}\label{def_P_sets}
  We define $\mathbb{P}^\forall_\phi$ and $\mathbb{P}^\exists_\phi$ by induction on $\phi$:
  \begin{itemize}
  \item When $\phi$ is atomic, $\mathbb{P}^\forall_\phi=\mathbb{P}^\exists_\phi$ is the set with a single point denoted $\ast$.
  \item When $\phi$ is $\neg\psi$, we define:
    \begin{itemize}
    \item $\mathbb{P}^\exists_\phi=\mathbb{P}^\forall_\psi$,
    \item $\mathbb{P}^\forall_\phi$ consists of partial functions $f$ from $\mathbb{P}^\exists_\phi=\mathbb{P}^\forall_\psi$ to $\mathbb{P}^\exists_\psi$ such that:
      \begin{itemize}
      \item the domain of $f$ is finite,
      \item if $a\in\dom(f)$, $a'\subseteq a$, and $a'\in\dom(f)$ then $f(a')\subseteq f(a)$,
      \item if $a\in\dom(f)$ then, for all $a'\leq a$, $a'\in\dom(f)$ and there is a $b\subseteq f(a')$ with $b\leq f(a)$,
      \end{itemize}
    \item The extension on $\mathbb{P}^\forall_\phi$ is given by $f\subseteq g$ if $\dom(f)\subseteq \dom(g)$ and, for all $a\in\dom(f)$, $f(a)\subseteq g(a)$.
    \item The ordering on $\mathbb{P}^\forall_\phi$ is given by $f\leq g$ if $\dom(f)=\dom(g)$ and for all $a\in\dom(f)$, $f(a)\leq g(a)$.
    \end{itemize}
  \item When $\phi$ is $\forall x\psi$ we define:
    \begin{itemize}
    \item $\mathbb{P}^\forall_\phi=\mathbb{P}^\forall_\psi$ and $\mathbb{P}^\exists_\phi=\mathbb{P}^\exists_\psi$,
    \item the orderings $\leq$ and $\subseteq$ are the same as for $\psi$.
    \end{itemize}
  \item When $\phi$ is $\bigwedge_{i\in I}\psi_i$ (where $I=\{0,1\}$ or $I=\mathbb{N}$), we define:\footnote{The interaction of $\subseteq$ and $\leq$ in the definition of $\mathbb{P}^\forall_\phi$ be may surprising.  A first attempt at an element of $\mathbb{P}^\forall_\phi$ would be a pair $(n,a)$ where $a\in\mathbb{P}^{\forall}_{\psi_n}$; importantly, $n$ is the ``bound'' corresponding to $\bigwedge$.  In order to get the necessary monotonicity properties, this needs to be extended to a pair $(n,\{a_i\}_{i\leq n})$ where each $a_i\in\mathbb{P}^\forall_{\psi_i}$.  The definition given here is analogous, using a single function $f$ where $\dom(f)$ stands in for $n$ and $f(i)=a_i$.  Then $f\leq g$ corresponds to $\dom(f)\subseteq\dom(g)$, because a larger domain corresponds to a larger value of $n$---that is, a larger bound.  On the other hand, $f\subseteq g$ should only happen when $\dom(f)=\dom(g)$---when the bound is the same.}
    \begin{itemize}
    \item elements of $\mathbb{P}^\exists_\phi$ are finite subsets $f$ of $\bigcup_{i\in I}\mathbb{P}^\exists_{\psi_i}$ such that $|f\cap \mathbb{P}^\exists_{\psi_i}|\leq 1$ for all $i$,
    \item if $f,g\in\mathbb{P}^\exists_\phi$ then $f\subseteq g$ exactly when, for each $e\in f$, there is an $e'\in g$ with $e\subseteq e'$,
    \item if $f,g\in\mathbb{P}^\exists_\phi$ then $f\leq g$ exactly when, for each $e\in f$, there is an $e'\in g$ with $e\leq e'$, and for each $e'\in g$ there is an $e\in f$ with $e\leq e'$,
    \item $\mathbb{P}^\forall_\phi$ consists of functions $f$ whose domain is a finite subset of $I$ so that $f(i)\in\mathbb{P}^\forall_{\psi_i}$ for each $i\in\dom(f)$,
    \item if $f,g\in\mathbb{P}^\forall_\phi$ then $f\subseteq g$ exactly when $\dom(f)=\dom(g)$ and for each $i\in\dom(f)$, $f(i)\subseteq g(i)$,
    \item if $f,g\in\mathbb{P}^\forall_\phi$ then $f\leq g$ exactly when $\dom(f)\subseteq\dom(g)$ and for each $i\in\dom(f)$, $f(i)\leq g(i)$.
    \end{itemize}
  \end{itemize}
\end{definition}

One crucial point of this definition is the following lemma, which is proven by a straightforward induction on $\phi$.
\begin{lemma}
  For any $f\in\mathbb{P}^Q_\phi$, the set of $g\in\mathbb{P}^Q_\phi$ such that $g\leq f$ is finite.
\end{lemma}

Almost all proofs and constructions involving the sets $\mathbb{P}^Q_\phi$ will be by induction on $\phi$; the case where $\phi$ is atomic is typically trivial and the case where $\phi$ is $\forall x\psi$ or $\exists x\psi$ will typically follow immediately from the inductive hypothesis, as will the case where $\phi$ is $\neg\psi$ and $Q$ is $\exists$.  The case where $\phi$ is $\bigwedge_{i\in I}\psi_i$ will typically either be a simplification of the $\neg\psi$ case or follow directly from the inductive hypothesis.  Therefore in the proofs below, when we note that the proof is by induction we will turn immediately to the main case, where $\phi$ is $\neg\psi$ and $Q$ is $\forall$.

\begin{lemma}
  For any $\phi$ and any $f,g\in\mathbb{P}^Q_\phi$, if $f\leq g$ and $g\leq f$ then $f=g$.
\end{lemma}
\begin{proof}
  By induction on $\phi$.  

Suppose $\phi$ is $\neg\psi$.  Let $f$ and $g$ be be elements of $\mathbb{P}^\forall_\phi$ with $f\leq g$ and $g\leq f$.  Then $\dom(f)=\dom(g)$ and, for every $a$ in this common domain, $f(a)\leq g(a)$ and $g(a)\leq f(a)$ so, by the inductive hypothesis, $f(a)=g(a)$.  Therefore $f=g$.
\end{proof}


\begin{definition}
  For any $\phi$, if $f,g_0,g_1\in\mathbb{P}^Q_\phi$ and both $g_0\leq f$ and $g_1\leq f$, we define an element $\min\{g_0,g_1\}\in\mathbb{P}^Q_\phi$ recursively by:
  \begin{itemize}
  \item When $\phi$ is atomic, $\min\{g_0,g_1\}=g_0=g_1=\ast$.
  \item When $\phi$ is $\neg\psi$:
    \begin{itemize}
    \item if $Q$ is $\exists$ then the definition is the same as for $\mathbb{P}^{\forall}_\psi$,
    \item if $Q$ is $\forall$ then $\min\{g_0,g_1\}$ is the function with domain $\dom(f)$ such that, for each $a\in\dom(f)$, $\min\{g_0,g_1\}(a)=\min\{g_0(a),g_1(a)\}$.
    \end{itemize}
  \item When $\phi$ is $\forall x\psi$ the definition is the same as for $\psi$.
  \item When $\phi$ is $\bigwedge_{i\in I}\psi_i$:
    \begin{itemize}
    \item if $Q$ is $\exists$ then $\min\{g_0,g_1\}=\{\min\{e_0,e_1\}\mid \text{there is some }e\in f\text{ such that }e_0\leq e\text{ and }e_1\leq e\}$,
    \item if $Q$ is $\forall$ then $\dom(\min\{g_0,g_1\})=\dom(g_0)\cap\dom(g_1)$ and, for each $i\in\dom(\min\{g_0,g_1\})$, $\min\{g_0,g_1\}(i)=\min\{g_0(i),g_1(i)\}$.
  \end{itemize}
  \end{itemize}
\end{definition}

The next lemma verifies that $\min$ really does find the greatest lower bound and that taking minimums respects both the $\subseteq$ and $\leq$ orderings.
\begin{lemma}
  For any $\phi$ and any $f,g_0,g_1\in\mathbb{P}^Q_\phi$ such that $g_0\leq f$ and $g_1\leq f$, the function $\min$ satisfies the following properties:
  \begin{enumerate}
  \item $\min\{g_0,g_1\}\in\mathbb{P}^Q_\phi$,
  \item $\min\{g_0,g_1\}\leq g_0$,
  \item $\min\{g_0,g_1\}\leq g_1$,
  \item if $h\in\mathbb{P}^Q_\phi$, $h\leq g_0$, and $h\leq g_1$ then $h\leq\min\{g_0,g_1\}$,
  \item for any $f',g'_0,g'_1\in\mathbb{P}^Q_\phi$ such that $g'_0\leq f'$ and $g'_1\leq f'$, if $f'\subseteq f$, $g'_0\subseteq g_0$, and $g'_1\subseteq g_1$ then $\min\{g'_0,g'_1\}\subseteq\min\{g_0,g_1\}$,
  \item for any $f',g'_0,g'_1\in\mathbb{P}^Q_\phi$ such that $g'_0\leq f'$ and $g'_1\leq f'$, if $f'\leq f$, $g'_0\leq g_0$, and $g'_1\leq g_1$ then $\min\{g'_0,g'_1\}\leq\min\{g_0,g_1\}$.
\end{enumerate}
\end{lemma}
\begin{proof}
  By induction on $\phi$.  

Suppose $\phi$ is $\neg\psi$.

(1):  Let $f,g_0,g_1\in\mathbb{P}^\forall_\phi$ be given with $g_0\leq f$ and $g_1\leq f$.  The inductive hypothesis guarantees that $\min\{g_0,g_1\}$ is a partial function from $\mathbb{P}^\forall_\psi$ to $\mathbb{P}^\exists_\psi$ with a finite domain.  To check that $\min\{g_0,g_1\}\in\mathbb{P}^\forall_\phi$, we must check the monotonicity conditions.

  Let $a,a'\in\dom(f)$ with $a'\subseteq a$.  Then $g_0(a')\subseteq g_0(a)$ and $g_1(a')\subseteq g_1(a)$ so, by the inductive hypothesis, $\min\{g_0(a'),g_1(a')\}\subseteq\min\{g_0(a),g_1(a)\}$.

  Let $a,a'\in\dom(f)$ with $a'\leq a$.  Then there are $b_0\subseteq g_0(a'), b_1\subseteq g_1(a')$ with $b_0\leq g_0(a), b_1\leq g_1(a)$.  Then, by the inductive hypothesis, $\min\{b_0,b_1\}\subseteq \min\{g_0(a'),g_1(a')\}=\min\{g_0,g_1\}(a')$ and since $\min\{b_0,b_1\}\leq b_0\leq g_0(a)$ and $\min\{b_0,b_1\}\leq b_1\leq g_1(a)$, we have $\min\{b_0,b_1\}\leq \min\{g_0,g_1\}(a)$ as needed. 

(2) and (3) are immediate from the definition.

(4): Suppose $h\leq g_0$ and $h\leq g_1$.  Then for each $a\in\dom(f)=\dom(h)$, we have $h(a)\leq g_0(a)$ and $h(a)\leq g_1(a)$, so $h(a)\leq \min\{g_0,g_1\}(a)$ by the inductive hypothesis.  Therefore $h\leq \min\{g_0,g_1\}$.

(5): Let $f',g'_0,g'_1\in\mathbb{P}^\forall_\phi$ with $g'_0\leq f'$ and $g'_1\leq f'$, and all of $f'\subseteq f$, $g'_0\subseteq g_0$, and $g'_1\subseteq g_1$.  Then for any $a\in\dom(f')=\dom(\min\{g'_0,g'_1\})\subseteq\dom(\min\{g_0,g_1\})$, we have $\min\{g'_0,g'_1\}(a)=\min\{g'_0(a),g'_1(a)\}\subseteq\min\{g_0(a),g_1(a)\}=\min\{g_0,g_1\}(a)$ by the inductive hypothesis.

(6): Similarly, let $f',g'_0,g'_1\in\mathbb{P}^\forall_\phi$ with $g'_0\leq f'$ and $g'_1\leq f'$, and all of $f'\leq f$, $g'_0\leq g_0$, and $g'_1\leq g_1$.  Then for any $a\in\dom(f')=\dom(\min\{g'_0,g'_1\})=\dom(\min\{g_0,g_1\})$, we have $\min\{g'_0,g'_1\}(a)=\min\{g'_0(a),g'_1(a)\}\leq\min\{g_0(a),g_1(a)\}=\min\{g_0,g_1\}(a)$ by the inductive hypothesis.
\end{proof}

We next define a restriction operation: given $f'\leq f$ and some $f^*\subseteq f$, we define $f'\upharpoonright f^*$ to be element which ``contains the part of $f'$ which is comparable to the information in $f^*$''.
\begin{definition}
  For any $\phi$, if $f,f',f^*\in\mathbb{P}^Q_\phi$ with $f'\leq f$ and $f^*\subseteq f$ then we define an element $f'\upharpoonright f^*\in\mathbb{P}^Q_\phi$ recursively by:
  \begin{itemize}
  \item When $\phi$ is atomic, $f'\upharpoonright f^*=\ast$.
  \item When $\phi$ is $\neg\psi$:
    \begin{itemize}
    \item if $Q$ is $\exists$ then the definition is the same as for $\mathbb{P}^{\forall}_\psi$,
    \item if $Q$ is $\forall$ then $f'\upharpoonright f^*$ is the function with domain $\dom(f^*)$ such that, for each $a\in\dom(f^*)$, $(f'\upharpoonright f^*)(a)=(f'(a))\upharpoonright f^*(a)$.
    \end{itemize}
  \item When $\phi$ is $\forall x\psi$ the definition is the same as for $\psi$.
  \item When $\phi$ is $\bigwedge_{i\in I}\psi_i$:
    \begin{itemize}
    \item if $Q$ is $\exists$ then $f'\upharpoonright f^*=\{e'\upharpoonright e^*\mid e'\in f', e^*\in f^*\text{ and there is an }e\in f\text{ such that }e'\leq e\text{ and }e^*\subseteq e\}$,
    \item if $Q$ is $\forall$ then $\dom(f'\upharpoonright f^*)=\dom(f')$ and, for each $i\in\dom(f')$, $(f'\upharpoonright f^*)(i)=f'(i)\upharpoonright f^*(i)$.
    \end{itemize}
  \end{itemize}
\end{definition}

The first three parts of the next lemma verify that we have defined $f'\upharpoonright f^*$ correctly: that it is a partial strategy contained in $f'$ and comparable to $f'$; the fourth part shows that it is the only such partial strategy. The last two parts establish that $\upharpoonright$ commutes with $\subseteq$ and $\leq$.
\begin{lemma}
  For any $\phi$ and any $f,f',f^*\in\mathbb{P}^Q_\phi$ such that $f'\leq f$ and $f^*\subseteq f$, the function $\upharpoonright$ satisfies the following properties:
  \begin{enumerate}
  \item $f'\upharpoonright f^*\in\mathbb{P}^Q_\phi$,
  \item $f'\upharpoonright f^*\leq f^*$,
  \item $f'\upharpoonright f^*\subseteq f'$,
  \item if $g\leq f^*$ and $g\subseteq f'$ then $g=f'\upharpoonright f^*$,
  \item for any $g,g',g^*\in\mathbb{P}^Q_\phi$ such that $g'\leq g$ and $g^*\subseteq g$, if all of $g\subseteq f$, $g'\subseteq f'$, $g^*\subseteq f^*$ then $(g'\upharpoonright g^*)\subseteq(f'\upharpoonright f^*)$,
  \item for any $g,g',g^*\in\mathbb{P}^Q_\phi$ such that $g'\leq g$ and $g^*\subseteq g$, if all of $g\leq f$, $g'\leq f'$, $g^*\leq f^*$ then $(g'\upharpoonright g^*)\leq(f'\upharpoonright f^*)$.
  \end{enumerate}
\end{lemma}
\begin{proof}
    By induction on $\phi$. 

Suppose $\phi$ is $\neg\psi$.

    (1): Let $f,f',f^*\in\mathbb{P}^\forall_\phi$ with $f'\leq f$ and $f^*\subseteq f$.  The inductive hypothesis guarantees that $f'\upharpoonright f^*$ is a partial function from $\mathbb{P}^\forall_\psi$ to $\mathbb{P}^\exists_\psi$ with a finite domain.  To check that $f'\upharpoonright f^*\in\mathbb{P}^\forall_\phi$, we must check the monotonicity conditions.

    Let $a,a'\in\dom (f'\upharpoonright f^*)=\dom(f^*)\subseteq\dom(f')$ with $a'\subseteq a$.  Then $f(a')\subseteq f(a)$, $f'(a')\subseteq f'(a)$, and $f^*(a')\subseteq f^*(a)$, so by the inductive hypothesis, $(f'\upharpoonright f^*)(a')=f'(a')\upharpoonright f^*(a')\subseteq f'(a)\upharpoonright f^*(a)=(f'\upharpoonright f^*)(a)$.

    The second monotonicity requirement requires more manipulation.  Let $a,a'\in\dom (f'\upharpoonright f^*)=\dom(f^*)\subseteq\dom(f')$ with $a'\leq a$.  Then there are a $b\subseteq f(a')$ with $b\leq f(a)$, a $b'\subseteq f'(a')$ with $b'\leq f(a')$, and a $b^*\subseteq f^*(a')$ with $b^*\leq f^*(a)$. Since $b\subseteq f(a')$ while $f'(a')\leq f(a')$, we know that $f'(a')\upharpoonright b\leq b\leq f(a)$.  Since $b'\leq f'(a)\leq f(a)$, there is an element $\min\{f'(a')\upharpoonright b,b'\}$.  Since both $f'(a')\upharpoonright b\subseteq f'(a')$ and $b'\subseteq f'(a')$, we have $\min\{f'(a')\upharpoonright b,b'\}\subseteq\min\{f'(a'),f'(a')\}=f'(a')$.  So we have both $\min\{f'(a')\upharpoonright b,b'\}\subseteq f'(a')$ and $\min\{f'(a')\upharpoonright b,b'\}\leq b$.  Therefore, withou generality, we may assume that $b'=f'(a')\upharpoonright b$---in particular, we may assume that $b'\leq b$.

    Similarly, since $f^*(a)\subseteq f(a)$ and $b\leq f(a)$, we have $b\upharpoonright f^*(a)\leq f^*(a)$.  Since $b^*\leq f^*(a)$, we have $\min\{b\upharpoonright f^*(a),b^*\}\leq f^*(a)$ and since $b\upharpoonright f^*\subseteq b\subseteq f(a')$ and $b^*\subseteq f^*(a')\subseteq f(a')$, we have $\min\{b\upharpoonright f^*(a),b^*\}\subseteq\min\{f(a'),f(a')\}=f(a')$.  Therefore $\min\{b\upharpoonright f^*(a),b^*\}=f(a')\upharpoonright b^*$.  But $b^*\leq b^*$ and $b^*\subseteq f(a')$, so $b^*$ itself is $f(a')\upharpoonright b^*$, so $b^*=\min\{b\upharpoonright f^*(a),b^*\}\subseteq b$.

    So we can consider $b'\upharpoonright b^*$: since $b\subseteq f(a')$, $b'\subseteq f'(a')$, and $b^*\subseteq f^*(a')$, we have $(b'\upharpoonright b^*)\subseteq (f'(a')\upharpoonright f^*(a'))$ and since $b\leq f(a)$, $b'\leq f'(a)$, and $b^*\leq f^*(a)$, we have $(b'\upharpoonright b^*)\leq (f'(a)\upharpoonright f^*(a))$.

(2) and (3) are immediate from the definition.

(4): Suppose $g$ is some function with $g\subseteq f'$ and $g\leq f^*$.  Then $\dom(g)=\dom(g^*)=\dom(f'\upharpoonright f^*)$ and for each $a\in\dom(g)$ we have $g(a)\subseteq f'(a)$ and $g(a)\leq f^*(a)$, so by the inductive hypothesis, $g(a)=(f'\upharpoonright f^*)(a)$.

(5): Suppose we also have $g,g',g^*$ with $g'\leq g$, $g^*\subseteq g$, and all of $g\subseteq f$, $g'\subseteq f'$, and $g^*\subseteq f^*$.   $\dom(g'\upharpoonright g^*)=\dom(g^*)\subseteq\dom(f^*)=\dom(f'\upharpoonright f^*)$ and, for each $a$ in this domain, $(g'\upharpoonright g^*)(a)\subseteq (f'\upharpoonright f^*)(a)$ by the inductive hypothesis.

(6): Suppose we also have $g,g',g^*$ with $g'\leq g$, $g^*\subseteq g$, and all of $g\leq f$, $g'\leq f'$, and $g^*\leq f^*$.  Then $\dom(g'\upharpoonright g^*)=\dom(g^*)=\dom(f^*)=\dom(f'\upharpoonright f^*)$ and, for each $a$ in this domain, $(g'\upharpoonright g^*)(a)\leq (f'\upharpoonright f^*)(a)$ by the inductive hypothesis.
\end{proof}
Note that (3) of this lemma implies that if $f'\leq f$ and $h\subseteq g\subseteq f$ then $(f'\upharpoonright g)\upharpoonright h=f'\upharpoonright h$.

\begin{definition}
  For any $\phi$ and any subset $\mathcal{F}\subseteq\mathbb{P}^Q_\phi$, we define when $\mathcal{F}$ is \emph{coherent} by:
  \begin{itemize}
  \item When $\phi$ is atomic, $\mathcal{F}$ is always coherent.
  \item When $\phi$ is $\neg\psi$:
    \begin{itemize}
    \item if $Q$ is $\exists$ then the definition is the same as for $\mathbb{P}^{\forall}_\psi$,
    \item if $Q$ is $\forall$ then $\mathcal{F}\subseteq\mathbb{P}^\forall_\phi$ is coherent if
      \begin{itemize}
      \item for any $a\in\bigcup_{f\in \mathcal{F}}\dom(f)$, $\{a\}$ is coherent,
      \item whenever $\mathcal{A}\subseteq\bigcup_{f\in \mathcal{F}}\dom(f)$ is finite and coherent, $\{f(a)\}_{a\in \mathcal{A}, f\in \mathcal{F}, a\in\dom(\mathcal{F})}$ is also coherent.
      \end{itemize}
    \end{itemize}
  \item When $\phi$ is $\forall x\psi$ the definition is the same as for $\psi$.
  \item When $\phi$ is $\bigwedge_{i\in I}\psi_i$:
    \begin{itemize}
    \item if $Q$ is $\exists$ then $\mathcal{F}$ is coherent if for each $i$, $\bigcup\mathcal{F}\cap\mathbb{P}^\exists_{\psi_i}$ is coherent,
    \item if $Q$ is $\forall$ then $\mathcal{F}$ is coherent if, for all $f,f'\in\mathcal{F}$, $\dom(f)=\dom(f')$, and for each $i$ in the common domain, $\{f(i)\}_{f\in\mathcal{F}}$ is coherent.
    \end{itemize}
  \end{itemize}
 
  For any $\phi$ and any finite, coherent, non-empty subset $\mathcal{F}\subseteq\mathbb{P}^Q_\phi$, we define an element $\bigcup\mathcal{F}$ by:
  \begin{itemize}
  \item When $\phi$ is atomic, $\bigcup\mathcal{F}=\ast$.
  \item When $\phi$ is $\neg\psi$:
    \begin{itemize}
    \item if $Q$ is $\exists$ then the definition is the same as for $\mathbb{P}^{\forall}_\psi$,
    \item if $Q$ is $\forall$ then $\dom(\bigcup \mathcal{F})=\bigcup_{f\in \mathcal{F}}\dom(f)$, and $(\bigcup \mathcal{F})(a)=\bigcup \{f(b)\}_{b\subseteq a, f\in \mathcal{F}, b\in\dom(f)}$.
    \end{itemize}
  \item When $\phi$ is $\forall x\psi$ the definition is the same as for $\psi$.
  \item When $\phi$ is $\bigwedge_{i\in I}\psi_i$:
    \begin{itemize}
    \item if $Q$ is $\exists$ then $\mathcal{F}=\{\bigcup\mathcal{F}\cap\mathbb{P}^\exists_{\psi_i}\mid \bigcup\mathcal{F}\cap\mathbb{P}^\exists_{\psi_i}\text{ is non-empty}\}$,
    \item if $Q$ is $\forall$ then $\dom(\bigcup\mathcal{F})$ is the domain of any (and therefore every) element of $\mathcal{F}$ and $(\bigcup\mathcal{F})(i)=\bigcup\{f(i)\}_{f\in\mathcal{F}}$.
    \end{itemize}
  \end{itemize}
\end{definition}

\begin{lemma}
  For any $\phi$ and any finite coherent $\mathcal{F}\subseteq\mathbb{P}^Q_\phi$,
  \begin{enumerate}
  \item $\bigcup \mathcal{F}\in\mathbb{P}^Q_\phi$,
  \item if $f\in\bigcup \mathcal{F}$ then $f\subseteq \bigcup \mathcal{F}$,
  \item if $f\in\mathbb{P}^Q_\phi$ and $\{f\}$ is coherent then $\bigcup\{f\}=f$,
  \item if $\mathcal{G}\subseteq\mathbb{P}^Q_\phi$ and there is a function $\sigma: \mathcal{G}\rightarrow \mathcal{F}$ such that $g\subseteq\sigma(g)$ for all $g\in \mathcal{G}$ then $\mathcal{G}$ is coherent and $\bigcup \mathcal{G}\subseteq \bigcup \mathcal{F}$,
  \item if $\mathcal{G}\subseteq\mathbb{P}^Q_\phi$ is coherent and there is a relation $\pi\subseteq \mathcal{F}\times \mathcal{G}$ such that
    \begin{itemize}
    \item  for each $f\in \mathcal{F}$ there is a $g\in \mathcal{G}$ with $(f,g)\in \pi$,
    \item  for each $g\in \mathcal{G}$ there is an $f\in \mathcal{F}$ with $(f,g)\in \pi$,
    \item if $(f,g)\in\pi$ then $g\leq f$,
    \end{itemize}
    then $\bigcup\mathcal{G}\leq\bigcup\mathcal{F}$.
  \end{enumerate}
\end{lemma}
\begin{proof}
  By induction on $\phi$.

Suppose $\phi$ is $\neg\psi$.  Let a finite coherent set $\mathcal{F}\subseteq\mathbb{P}^\forall_\phi$ be given.

  (1): First, note that for any $a\in\bigcup_{f\in \mathcal{F}}\dom(f)$, since $\{a\}$ is coherent, also $\{b\}_{b\subseteq a,b\in\bigcup_{f\in F}\dom(f)}$ is coherent and $\{b\}_{b\subseteq a,b\in\bigcup_{f\in \mathcal{F}}\dom(f)}\subseteq a$ by the inductive hypothesis, so $\{f(b)\}_{b\subseteq a,f\in \mathcal{F},b\in\dom(f)}$ is coherent as well.

As usual, we need to check the monotonicty requirements.  Let $a,a'\in\dom(\mathcal{F})$ with $a'\subseteq a$.  Let $\mathcal{G}=\{f(b)\}_{b\subseteq a,f\in \mathcal{F},b\in\dom(f)}$ and let $\mathcal{G}'=\{f(b)\}_{b\subseteq a',f\in \mathcal{F},b\in\dom(f)}$.  Then $\mathcal{G}'\subseteq\mathcal{G}$, so the identity maps $\mathcal{G}'$ to $\mathcal{G}$ and the inductive hypothesis ensures that $(\bigcup\mathcal{F})(a')=\bigcup\mathcal{G}'\subseteq\bigcup\mathcal{G}=(\bigcup\mathcal{F})(a)$.

Let $a,a'\in\dom(\mathcal{F})$ with $a'\leq a$.  Let $\mathcal{G}=\{f(b)\}_{b\subseteq a,f\in \mathcal{F},b\in\dom(f)}$ and let $\mathcal{G}'=\{f(b)\}_{b\subseteq a',f\in \mathcal{F},b\in\dom(f)}$.  Whenever $b\subseteq a$, $f'\in\mathcal{F}$, and $b\in\dom(f)$, we have $a'\upharpoonright b\leq b$, so there is a $d'\subseteq f(a'\upharpoonright b)$ with $d'\leq f(b)$; we let $\mathcal{H}$ be the set of such $d'$.

We define $\pi\subseteq\mathcal{G}\times\mathcal{H}$ by placing $(d,d')\in \pi$ if there is a $b\subseteq a$ and an $f\in\mathcal{F}$ so that $b\in\dom(f)$, $d=f(b)$, and $d'\subseteq f(a'\upharpoonright b)$ satisfies $d'\leq f(b)=d$, so whenever $(d,d')\in\pi$, we have $d'\leq d$.  By definition, for any $d'\in\mathcal{H}$ we have a $d\in\mathcal{G}$ with $(d,d')\in\pi$.  Conversely, if $d\in\mathcal{G}$, so $d=f(b)$ for some $f\in\mathcal{F}$ and $b\subseteq a$ with $b\in\dom(f)$, we have a $d'\subseteq f(a'\upharpoonright b)$ with $(d,d')\in\pi$.

We define a function $\sigma:\mathcal{H}\rightarrow\mathcal{G}'$: given any $d'\in\mathcal{H}$, we choose some $b,f$ so that $d'\subseteq f(a'\upharpoonright b)$ and set $\sigma(d')=f(a'\upharpoonright b)$, so $d'\subseteq \sigma(d')$.  Since $\mathcal{G}'$ is coherent, by the inductive hypothesis, so is $\mathcal{H}$, and $\bigcup\mathcal{H}\subseteq\bigcup\mathcal{G}'=(\bigcup\mathcal{F})(a')$.  Then, by the inductive hypothesis again, $\bigcup\mathcal{H}\leq\bigcup\mathcal{G}=(\bigcup\mathcal{F})(a)$.

  (2) and (3) are immediate from the definition.

  (3): Let $G\subseteq \mathbb{P}^\forall_\phi$ and $\sigma: \mathcal{G}\rightarrow \mathcal{F}$ be given.  For any $a\in\bigcup_{g\in \mathcal{G}}\dom(g)$, we have $a\in\dom(g)$ and so $a\in\dom(\sigma(g))$, so $\{a\}$ is coherent.  For any coherent $A\subseteq \bigcup_{g\in \mathcal{G}}\dom(g)$, we similarly have $A\subseteq\bigcup_{f\in \mathcal{F}}\dom(g)$, so $\mathcal{F}'=\{f(a)\}_{a\in A, f\in \mathcal{F}, a\in\dom(f)}$ is coherent.  Consider $\mathcal{G}'=\{g(a)\}_{a\in A, g\in \mathcal{G}, a\in\dom(g)}$; we may define $\sigma:\mathcal{G}'\rightarrow \mathcal{F}'$ by taking $\sigma(g(a))=(\sigma(g))(a)$ (choosing arbitrarily if there are multiple choices for $g,b$).  Then, by the inductive hypothesis, $\bigcup \mathcal{G}'$ is coherent and $\bigcup \mathcal{G}'\subseteq\bigcup \mathcal{F}'$.  Therefore $\bigcup \mathcal{G}\subseteq \bigcup \mathcal{F}$.

  (4): Let $\mathcal{G}\subseteq\mathbb{P}^\forall_\phi$ be coherent and let $\pi\subseteq\mathcal{F}\times \mathcal{G}$ be given.  Since whenever $(f,g)\in\pi$ we have $\dom(f)=\dom(g)$, we have $\dom(\bigcup\mathcal{G})=\dom(\bigcup\mathcal{F})$.  For any $a$ in this domain, let $\mathcal{F}'=\{f(b)\}_{b\subseteq a,f\in\mathcal{F},b\in\dom9f)}$ and $\mathcal{G}'=\{g(b)\}_{b\subseteq a,g\in\mathcal{G},b\in\dom(g)}$.  Define $\pi'\subseteq\mathcal{F}'\times\mathcal{G}'$ by $(d,d')\in\pi$ if there is some $b\subseteq a$ and some $(f,g)\in\pi$ with $d=f(b)$ and $d'=g(b)$.  Since $(f,g)\in\pi$, $g\leq f$, so $d'\leq d$.  For any $d\in\mathcal{F}'$, we may find some $f\in\mathcal{F}$ and some $b\subseteq a$ with $f(b)=d$, so some $g\in\mathcal{G}$ with $(f,g)\in\pi$, so $b\in\dom(g)$, so $(d,g(b))\in\pi$; a symmetric argument holds for $d'\in\mathcal{G}'$.  Therefore, by the inductive hypothesis, $(\bigcup \mathcal{G})(a)=\bigcup \mathcal{G}'\leq \bigcup \mathcal{F}'=(\bigcup \mathcal{F})(a)$.
\end{proof}

\begin{lemma}
  If $\mathcal{F}\subseteq \mathcal{G}\subseteq\mathbb{P}^Q_\phi$ and $\mathcal{G}$ is coherent then $\mathcal{F}$ is coherent.
\end{lemma}
\begin{proof}
  Straightforward induction.
\end{proof}

\begin{definition}\label{def_F_sets}
  For each $\phi$ and $Q$, we define $\mathbb{F}^Q_\phi\subseteq\mathbb{P}^Q_\phi$ to be those $f\in\mathbb{P}^Q_\phi$ such that $\{f\}$ is coherent.
\end{definition}
To see what we're ruling out, consider the case where $\mathbb{P}^Q_\phi$ is approximating functions from $\mathbb{N}^{\mathbb{N}}$ to $\mathbb{N}$. That is, $f\in\mathbb{P}^Q_\phi$ means that $\dom(f)$ is a collection of finite partial functions from $\mathbb{N}^{\mathbb{N}}$. There could be elements in the domain of $f$ which are comparable; for instance, consider $f=\{((1,3),3),((2,4),4)\}$; that is, whenever $g$ is a function with $g(1)=3$, $f(g)=3$, while whenever $g(2)=4$, $f(g)=4$. We can immediately see the problem: what if both $g(1)=3$ and $g(2)=4$? This is an incoherent partial strategy. The definition of coherence requires that whenever we have elements of $\dom(f)$ which are consistent, $f$ has to give consistent behavior on them.

\begin{lemma}
  If $f\in\mathbb{F}^Q_\phi$,  $f'\in\mathbb{P}^Q_\phi$, and $f'\subseteq f$ then $f'\in\mathbb{F}^Q_\phi$.
\end{lemma}
\begin{proof}
  Straightforward induction.
\end{proof}

\begin{lemma}
  If $\mathcal{F}\subseteq\mathbb{F}^Q_\phi$ is coherent then $\bigcup \mathcal{F}\in\mathbb{F}^Q_\phi$.
\end{lemma}
\begin{proof}
Straightforward induction.
\end{proof}

When working with elements of $\mathbb{F}^Q_{\neg\psi}$, it is convenient to replace coherent subsets of the domain with single elements.

\begin{definition}
  Given $f\in\mathbb{F}^Q_{\neg\psi}$, we define $\tilde f\supseteq f$ by:
\begin{itemize}
\item $\dom(\tilde f)$ consists of all $\bigcup\mathcal{A}$ such that $\mathcal{A}\subseteq\dom(f)$ is coherent,
\item if $a\in\dom(\tilde f)$, so $a=\bigcup\mathcal{A}$ for some coherent $\mathcal{A}\subseteq\dom(f)$, $\tilde f(a)=\bigcup\{f(b)\}_{b\subseteq a, b\in\dom(f)}$.
\end{itemize}
\end{definition}
We think of $\tilde f$ as a completition of $f$; when $\mathcal{A}\subseteq\dom(f)$ is coherent, coherence forces us to have a well defined value for $\tilde f(\bigcup\mathcal{A})$, so we may as well include it.

\begin{lemma}
\begin{enumerate}
\item If $f,g\in\mathbb{F}^Q_{\neg\psi}$ with $g\subseteq f$ then $\tilde g\subseteq\tilde f$.
\item If $f,g\in\mathbb{F}^Q_{\neg\psi}$ with $g\leq f$ then $\tilde g\leq\tilde f$.
\item Whenever $\mathcal{A}\subseteq\dom(\tilde f)$ is coherent, there is an $a\in\dom(\tilde f)$ such that $a=\bigcup\mathcal{A}$.
\item If $f,g\in\mathbb{F}^Q_{\neg\psi}$ and $g\leq \tilde f$ then $g=\tilde g$.
\end{enumerate}
\end{lemma}

\begin{lemma}

  If $f,g_0,g_1\in\mathbb{F}^Q_\phi$ and both $g_0\leq f$ and $g_1\leq f$ then $\min\{f_0,g_1\}\in\mathbb{F}^Q_\phi$.
\end{lemma}
\begin{proof}
  By induction on $\phi$.

  Suppose $f,g_0,g_1\in\mathbb{F}^\forall_{\neg\psi}$.  When $f=\tilde f$ (and therefore $g_0=\tilde g_0, g_1=\tilde g_1$) it is immediate from the inductive hypothesis that $\min\{g_0,g_1\}\in\mathbb{F}^\forall_{\neg\psi}$.  So given $f\subsetneq \tilde f$, observe that $\min\{g_0,g_1\}=\min\{\tilde g_0,\tilde g_1\}\upharpoonright f$, so the claim follows from the previous lemma.
\end{proof}

\begin{lemma}
    If $f,f',f^*\in\mathbb{F}^Q_\phi$ with $f'\leq f$ and $f^*\subseteq f$ then $f'\upharpoonright f^*\in\mathbb{F}^Q_\phi$.
\end{lemma}
\begin{proof}
    By induction on $\phi$.

    Suppose $f,f',f^*\in\mathbb{F}^\forall_{\neg\psi}$.  When $f^*=\tilde f^*$, it is immediate from the inductive hypothesis that $f'\upharpoonright f^*\in\mathbb{F}^Q_\phi$.  So suppose $f^*\subsetneq\tilde f^*$.  Then $f'\upharpoonright f^*=(\tilde f'\upharpoonright \tilde f^*)\upharpoonright f^*$, so for any $\mathcal{A}\subseteq\dom(f'\upharpoonright f^*)$, we have $(f'\upharpoonright f^*)(b)\subseteq (\tilde f'\upharpoonright \tilde f^*)(\bigcup\mathcal{A})$ for all $b\in\mathcal{A}$, so we certainly have that $\{(f'\upharpoonright f^*)(b)\}_{b\in\mathcal{A}}$ is coherent.
\end{proof}

\begin{lemma}\label{thm:coherent_extension}
Let $f,g,g'\in\mathbb{F}^Q_\phi$ be given with $g'\leq g\subseteq f$.  Let $\mathcal{K}\subseteq\mathbb{F}^Q_\phi$ be given (not necessarily coherent) so that for each $k\in \mathcal{K}$, $k\leq f$ and $k\upharpoonright g\leq g'$.


Then there is an $f'\leq f$ such that $g'\subseteq f'$ and, for each $k\in K$, $k\leq f'$.
\end{lemma}
\begin{proof}
  By induction on $\phi$.

  When $\phi$ is $\neg\psi$ and $Q$ is $\forall$, let $f, g,g',\mathcal{K}$ be given.  We claim that, without loss of generality, we may assume that $f=\tilde f$ (and therefore $k=\tilde k$ for all $k\in \mathcal{K}$).  For suppose we find $\tilde f'\leq \tilde f$ with each $g'\subseteq \tilde f'$ and each $\tilde k\leq \tilde f'$; then we may take $f'=\tilde f'\upharpoonright f$.

  So we assume that $f=\tilde f$.  We will construct $f'\supseteq g'$ by adding elements to the domain one by one, so it suffices to construct an $f'$ with $g'\subseteq f'$ and $\dom(f')=\dom(g')\cup\{a\}$ for some $a\in\dom(f)\setminus\dom(g')$ such that for any $a'\in\dom(f)$ with either $a'\subseteq a$ or $a'\leq a$, $a'\in\dom(g')$.

Let $\mathcal{B}$ be the set of $b$ such that $b\subseteq a$ and $b\in\dom(g)$.  If $\mathcal{B}$ is empty then we may simply set $f'(a)=f(a)$; this satisfies the monotonicity requirement since if $a'\in\dom(g')$ with $a'\leq a$, we have $g'(a')\leq f(a')$, and $f(a')\upharpoonright f(a)\leq f(a)$, so $g'(a')\upharpoonright f(a)\leq f(a)$ as well.

So suppose $\mathcal{B}$ is non-empty.  Let $\mathcal{D}=\{g(b)\}_{b\in \mathcal{B}}$ and $\mathcal{D}'=\{g'(b)\}_{b\in\mathcal{B}}$ and set $d=\bigcup \mathcal{D}$ and $d'=\bigcup\mathcal{D}'$.  Then $d'\leq d\subseteq f(a)$.  

Let $\mathcal{L}_0$ consist of all elements of the form $g'(a')\upharpoonright f(a)$ with $a'< a$ and let $\mathcal{L}_1$ consist of all $k(a')\upharpoonright f(a)$ with $a'\leq a$ and $k\in \mathcal{K}$.

We claim that $f(a)$, $d$, $d'$, $\mathcal{L}_0\cup \mathcal{L}_1$ satisfy the assumptions of the lemma so that we can apply the inductive hypothesis; the condition remaining to check is that, for any $\ell\in \mathcal{L}_0\cup \mathcal{L}_1$, $\ell\upharpoonright d\leq d'$.  First, suppose $\ell\in \mathcal{L}_0$, so $\ell=g'(a')\upharpoonright f(a)$.  It suffices to show that, for each $b\in\mathcal{B}$, $\ell\upharpoonright g'(b)\leq g'(b)$.  Then we have $\ell\upharpoonright g'(b)=g'(a')\upharpoonright g'(b)=g'(a'\upharpoonright b)\upharpoonright g'(b)\leq g'(b)$ as needed.

Next, suppose $\ell\in L_1$, so $\ell=k(a')\upharpoonright f(a)$.  We carry out a similar argument: $\ell\upharpoonright g'(b)=k(a')\upharpoonright g'(b)=(k\upharpoonright g')(a')\upharpoonright g'(b)\leq g'(a')\upharpoonright g'(b)\leq g'(b)$.

So the inductive hypothesis gives us a $b^*\leq f(a)$ with $g'(b)\subseteq b^*$ for all $b\in \mathcal{B}$ and $k(a)\leq b^*$ for all $k\in\mathcal{B}$, and we define $f'\supseteq g'$ by $f'(a)=b^*$.
\end{proof}

\subsection{Satisfaction}

Because we are dealing with partial functions, it is possible that attempting to evaluate these functions could lead us outside the domain.  We wish to define the case where enough values are defined for the process of bounding a formula to complete.

\begin{definition}
  When $a\in\mathbb{F}^\forall_\phi$ and $e\in\mathbb{F}^\exists_\phi$, we define when the pair $(a,e)$ is \emph{decisive for $\phi$} inductively:
  \begin{itemize}
  \item When $\phi$ is atomic, $(a,e)$ is always decisive for $\phi$.
  \item When $\phi$ is $\neg\psi$, $(a,e)$ is decisive for $\phi$ if for every $e'\leq e$ there is an $e^*\subseteq e'$ such that $(e',a(e^*))$ is decisive.
  \item When $\phi$ is $\forall x\psi$, $(a,e)$ is decisive for $\phi$ iff $(a,e)$ is decisive for $\psi$.
  \item When $\phi$ is $\bigwedge_{i\in I}\psi_i$, $(a,e)$ is decisive for $\phi$ iff for each $i\in\dom(a)$ there is an $r\in e\cap\mathbb{F}^\exists_{\psi_i}$ such that $(a(i),r)$ is decisive for $\psi_i$.
  \end{itemize}
\end{definition}

\begin{lemma}\label{lemma:decisive_montone}
If $(a,e)$ is decisive for $\phi$, $a'\leq a$, and $e'\leq e$ then $(a',e')$ is decisive for $\phi$ as well.
\end{lemma}
\begin{proof}
By induction on $\phi$.  

When $\phi$ is $\neg\psi$, for any $\hat e\leq e'\leq e$, there is an $e^*\subseteq \hat e$ such that $(e',a(e^*))$ is decisive.  Since $a'(e^*)\leq a(e^*)$, by the inductive hypothesis, also $(e',a'(e^*))$ is decisive as needed.

When $\phi$ is $\bigwedge_{i\in I}\psi_i$, we have $\dom(a')\subseteq\dom(a)$, so for each $i\in\dom(a')$ there is an $r\in e\cap\mathbb{F}^\exists_{\psi_i}$ such that $(a(i),r)$ is decisive for $\psi_i$; there is an $r'\in e'$ with $r'\leq r$, so by the inductive hypothesis, $(a'(i),r')$ is decisive for $\psi_i$ as well.
\end{proof}

\begin{lemma}
If $(a,e)$ is decisive for $\phi$, $a\subseteq a'$ and $e\subseteq e'$ then $(a',e')$ is decisive for $\phi$ as well.
\end{lemma}
\begin{proof}
By induction on $\phi$. 

When $\phi$ is $\neg\psi$, for every $\hat e\leq e'$, we have a unique $\hat e\upharpoonright e\leq e$, and therefore an $e^*\subseteq \hat e\upharpoonright e\subseteq \hat e$ such that $(\hat e\upharpoonright e,a(e^*))$ is decisive for $\psi$.  Since $a(e^*)\subseteq a'(e^*)$, also $(\hat e,a'(e^*))$ is decisive for $\psi$.

When $\phi$ is $\bigwedge_{i\in I}\psi_i$, we have we have $\dom(a')\subseteq\dom(a)$, so for each $i\in\dom(a')$ there is an $r\in e\cap\mathbb{F}^\exists_{\psi_i}$ such that $(a(i),r)$ is decisive for $\psi_i$; there is an $r'\in e'$ with $r\subseteq r'$, so by the inductive hypothesis, $(a'(i),r')$ is decisive for $\psi_i$ as well.
\end{proof}

\begin{definition}
  Suppose $a\in\mathbb{F}^\forall_\phi$ and $e\in\mathbb{F}^\exists_\phi$ and $(a,e)$ is decisive for $\phi$.  We define when $\mathfrak{M}\vDash^{\leq a,e}\phi[\vec b]$ holds recursively:
  \begin{itemize}
  \item If $\phi$ is atomic, $\mathfrak{M}\vDash^{\leq a,e}\phi[\vec b]$ iff $\mathfrak{M}\vDash\phi[\vec b]$.
  \item If $\phi$ is $\neg\psi$, $\mathfrak{M}\vDash^{\leq a,e}\phi[\vec b]$ holds iff there is an $e'\leq e$ and an $e^*\subseteq e'$ such that $(e',a(e^*))$ is decisive for $\psi$ and $\mathfrak{M}\not\vDash^{\leq e',a(e^*)}\psi[\vec b]$.
  \item If $\phi$ is $\forall x\psi$, $\mathfrak{M}\vDash^{\leq a,e}\phi[\vec b]$ iff for every $b'\in|\mathfrak{M}|$, $\mathfrak{M}\vDash^{\leq a,e}\psi[\vec b,x\mapsto b']$.
  \item If $\phi$ is $\bigwedge_{i\in I}\psi_i$, $\mathfrak{M}\vDash^{\leq a,e}\phi[\vec b]$ iff for each $i\in\dom(a)$, letting $r\in e\cap\mathbb{P}^\exists_{\psi_i}$, we have $\mathfrak{M}\vDash^{\leq a(i),r}\psi_i[\vec b]$.
  \end{itemize}
\end{definition}

\begin{lemma}
If $\{a_0,a_1\}$ and $\{e_0,e_1\}$ are both coherent and $(a_0,e_0)$ and $(a_1,e_1)$ are both decisive for $\phi$ then 
\[\mathfrak{M}\vDash^{\leq a_0,e_0}\phi[\vec b]\ \Leftrightarrow \mathfrak{M}\vDash^{\leq a_1,e_1}\phi[\vec b].\]
\end{lemma}
\begin{proof}
  By induction on $\phi$.  

  Suppose $\phi$ is $\neg\psi$.  If $\mathfrak{M}\vDash^{\leq a_0,e_0}\phi[\vec b]$ then there is an $e'_0\leq e_0$ and an $e^*_0\subseteq e'_0$ so that $\mathfrak{M}\not\vDash^{\leq e'_0,a_0(e^*_0)}\psi[\vec b]$.

  Let $e^+=\bigcup\{e_0,e_1\}$.  Then by Lemma \ref{thm:coherent_extension} there is an $e'\leq e^+$ with $e'_0\subseteq e'$, and therefore an $e'_1=e'\upharpoonright e_1\leq e_1$.  Since $(a_1,e_1)$ is decisive for $\phi$, there must be some $e^*_1\subseteq e'_1$ so that $(e'_1,a_1(e^*_1))$ is decisive for $\neg\psi$.  Since $\{a_0,a_1\}$ is coherent and $\{e^*_0,e^*_1\}$ is coherent, also $\{a_0(e^*_0),a_1(e^*_1)\}$ is coherent, and since also $\{e'_0,e'_1\}$ is coherent, by the inductive hypothesis
\[\mathfrak{M}\vDash^{\leq e'_0,a_0(e^*_0)}\psi[\vec b]\Leftrightarrow\mathfrak{M}\vDash^{\leq e'_1,a_1(e^*_1)}\psi[\vec b].\]
Therefore $\mathfrak{M}\not\vDash^{\leq e'_1,a_1(e^*_1)}\psi[\vec b]$, so $\mathfrak{M}\vDash^{\leq a_1,e_1}\phi[\vec b]$ as needed.
\end{proof}



\begin{lemma}\label{thm:fol}
  If $(a,e)$ is decisive for $\phi$ then there is a first-order formula $\hat\phi^{\leq a,e}$ such that $\mathfrak{M}\vDash^{\leq a,e}\phi[\vec b]$ iff $\mathfrak{M}\vDash\hat\phi^{\leq a,e}[\vec b]$.
\end{lemma}
\begin{proof}
  By induction on $\phi$.

  When $\phi$ is $\neg\psi$, we use the fact that there are only finitely many $e'\leq e$.  For each $e'\leq e$, let $e''\subseteq e'$ be such that $(e',a( e''))$ is decisive for $\psi$, so we may take $\hat\phi^{\leq a,e}$ to be
\[\bigvee_{e'\leq e}\neg\hat\psi^{\leq e',a( e'')}.\]
Suppose $\mathfrak{M}\vDash^{\leq a,e}\phi[\vec b]$.  Then there is some $e'\leq e$ and some $e^*\subseteq e$ so that $\mathfrak{M}\not\vDash^{\leq e',a(e^*)}\psi[\vec b]$, so by the inductive hypothesis, $\mathfrak{M}\vDash\neg\hat\psi^{\leq e',a(e^*)}[\vec b]$.  Since $e^*\subseteq e'$ and $e''\subseteq e'$, $\{e^*, e''\}$ is coherent, so $\{a(e^*),a( e'')\}$ is coherent, so also $\mathfrak{M}\vDash\neg\psi^{\leq e',a(e'')}[\vec b]$, so $\mathfrak{M}\vDash\hat\phi^{\leq a,e}[\vec b]$.

Conversely, if $\mathfrak{M}\vDash\hat\phi^{\leq a,e}[\vec b]$ then there is an $e'\leq e$ so that $\mathfrak{M}\vDash\neg\hat\psi^{\leq e',a(e'')}[\vec b]$, so by the inductive hypothesis, $\mathfrak{M}\not\vDash^{\leq e',a(e'')}\psi[\vec b]$, so $\mathfrak{M}\vDash^{\leq a,e}\phi[\vec b]$.
\end{proof}

\subsection{Strategies}

We want to extend these finite fragments to total ``strategies'' which are always decisive.  These strategies will be topological spaces $\mathbb{S}^Q_\phi$ with a family of basic open sets given by $\{U_f\mid f\in\mathbb{F}^Q_\phi\}$.  We will write $f\subseteq F$ if $F\in U_f$.
\begin{definition}\label{def:strategies}
  \begin{itemize}
  \item When $\phi$ is atomic, $\mathbb{S}^Q_\phi$ is a singleton set and the only basic open set is $U_\ast=\{\ast\}$ where $\ast$ is the unique element of both $\mathbb{S}^Q_\phi$ and $\mathbb{F}^Q_\phi$.
  \item When $\phi$ is $\neg\psi$:
    \begin{itemize}
    \item $\mathbb{S}^\exists_\phi=\mathbb{S}^\forall_\psi$ with the same topology.
    \item $\mathbb{S}^\forall_\phi$ is the set of functions $F$ from $\mathbb{S}^\exists_\phi=\mathbb{S}^\forall_\psi$ to $\mathbb{S}^\exists_\psi$ satisfying the following uniform continuity condition: for all $b$ and all $A$ with $F(A)\in U_b$, there is an $a$ with $A\in U_a$ such that, for each $a'\leq a$, there is a $b'\leq b$ with $F(U_{a'})\subseteq U_{b'}$.  For any $f\in\mathbb{F}^\forall_\phi$, $U_f$ is the set of $F$ such that, for every $a\in\dom(f)$ and any $A\in U_a$, $F(A)\in U_{f(a)}$.
    \end{itemize}
  \item When $\phi$ is $\forall x\psi$, $\mathbb{S}^Q_\phi$ is $\mathbb{S}^Q_\psi$.
  \item When $\phi$ is $\bigwedge_{i\in I}\psi_i$
    \begin{itemize}
    \item $\mathbb{S}^\exists_\phi$ consists of sets $F\subseteq\bigcup\mathbb{S}^\exists_{\psi_i}$ such that, for each $i$, $|F\cap\mathbb{S}^\exists_{\psi_i}|= 1$.  For any $f\in\mathbb{F}^\exists_\phi$, $U_f$ is the set of $F$ such that for each $e\in f$ there is an $E\in F\cap U_e$.
    \item $\mathbb{S}^\forall_\phi$ consists of functions $F$ whose domain is a finite subset of $I$ and $F(i)\in\mathbb{S}^\forall_{\psi_i}$ for all $i\in\dom(F)$.  For any $f\in\mathbb{F}^\forall_\phi$, $U_f$ is the set of $F$ such that $\dom(F)=\dom(f)$ and for every $i\in\dom(f)$, $F(i)\in U_{f(i)}$.
    \end{itemize}
  \end{itemize}
\end{definition}

\begin{definition}
  We define partial orderings on $\mathbb{S}^Q_\phi$ recursively by:
  \begin{itemize}
  \item When $\phi$ is atomic, the ordering is trivial.
  \item When $\phi$ is $\neg\psi$:
    \begin{itemize}
    \item When $Q$ is $\exists$, the definition is the same as for $\mathbb{P}^\forall_\psi$.
    \item When $Q$ is $\forall$, we say $G\leq F$ if for every $A\in\mathbb{S}^Q_\psi$, $G(A)\leq F(A)$.
    \end{itemize}
  \item When $\phi$ is $\forall x\psi$ the definition is the same as for $\mathbb{P}^Q_\psi$.
  \item When $\phi$ is $\bigwedge_{i\in I}\psi_i$''
    \begin{itemize}
    \item When $Q$ is $\exists$, $G\leq F$ if for each $E\in G$ there is an $E'\in F$ with $E\leq E'$.
    \item When $Q$ is $\forall$, $G\leq F$ if $\dom(G)\subseteq \dom(F)$ and for all $i\in \dom(G)$, $G(i)\leq F(i)$.
    \end{itemize}
  \end{itemize}
\end{definition}

\begin{definition}
  When $\mathcal{F}\subseteq\mathcal{S}^Q_\phi$ is non-empty and finite, we define a $\max\mathcal{F}\in\mathcal{S}^Q_\phi$ by:
  \begin{itemize}
  \item When $\phi$ is atomic, $\max\mathcal{F}=\ast$.
  \item When $\phi$ is $\neg\psi$:
    \begin{itemize}
    \item When $Q$ is $\exists$, the definition is the same as for $\mathbb{P}^\forall_\psi$.
    \item When $Q$ is $\forall$, $(\max\mathcal{F})(A)=\max\{F(A)\}_{F\in\mathcal{F}}$.
    \end{itemize}
  \item When $\phi$ is $\forall x\psi$ the definition is the same as for $\mathbb{P}^Q_\psi$.
  \item When $\phi$ is $\bigwedge_{i\in I}\psi_i$:
    \begin{itemize}
    \item When $Q$ is $\exists$, $\max\mathcal{F}$ is $\{\max(\bigcup\mathcal{F}\cap\mathbb{P}^\exists_{\psi_i})\}$.
    \item When $Q$ is $\forall$, $\dom(\max\mathcal{F})=\bigcup_{F\in\mathcal{F}}\dom(F)$ and $(\max\mathcal{F})(i)=\max\{F(i)\}_{F\in\mathcal{F}}$.
    \end{itemize}
  \end{itemize}
\end{definition}

\begin{lemma}
  For any non-empty finite $\mathcal{F}\subseteq\mathcal{S}^\forall_\phi$, $F\leq\max\mathcal{F}$ for all $F\in\mathcal{F}$.
\end{lemma}

\begin{lemma}
  $\mathcal{F}\subseteq\mathbb{P}^Q_\phi$ is coherent if and only if every finite subset of $\mathcal{F}$ is coherent.
\end{lemma}

\begin{lemma}
  The sets $U_f$ are non-empty.
\end{lemma}
\begin{proof}
  By induction on $\phi$.

Suppose $\phi$ is $\neg\psi$ and let $f\in\mathbb{F}^\forall_\phi$.  Since $f\subseteq \tilde f$, we have $U_{\tilde f}\subseteq U_f$, so we may assume without loss of generality that $f=\tilde f$.  Order $\dom(f)$ so that if $a'\subseteq a$ then $a\prec a'$ and if $a\leq a'$ then $a\prec a'$; by induction along $\prec$, for each $a\in\dom(f)$ choose a $B_a\in U_{f(a)}$ such that if $a'\in\dom(f)$, $a^*\subseteq a'$, and $a^*\leq a$ then $B_{a'}\leq B_a$.  (Since maxima exist over finite sets by the previous lemma, such elements always exist.)  We define $F\in\mathbb{S}^\forall_\phi$ by taking, for each $A\in\mathbb{S}^\forall_{\psi}$, the longest (i.e. $\subseteq$-maximal) $a\in\dom(f)$ with $A\in U_f$, and setting $F(A)=B_a$.  Such a function is automatically continuous and satisfies the monotonicity requirement.
\end{proof}

\begin{definition}
  For any $F\in\mathbb{S}^Q_\phi$, $\mathcal{C}(F)$ is the set of $f\in\mathbb{F}^Q_\phi$ such that $f\subseteq F$.
\end{definition}

\begin{lemma}
  If $F\in\mathbb{S}^Q_\phi$ and $\mathcal{A}\subseteq\mathcal{C}(F)$ is finite then $\mathcal{A}$ is coherent and $\bigcup\mathcal{A}\subseteq F$.
\end{lemma}
\begin{proof}
  Straightforward induction.
\end{proof}

\begin{lemma}
  $\mathcal{C}(F)$ is a maximal coherent set.
\end{lemma}
\begin{proof}
  By induction on $\phi$.

  Suppose $\phi$ is $\neg\psi$ and $Q$ is $\forall$.  For any $a\in\bigcup_{f\in \mathcal{C}(F)}\dom(f)$, $\{a\}$ is automatically coherent because $f\in\mathbb{F}^\forall_\phi$.  If $\mathcal{A}\subseteq\bigcup_{f\in\mathcal{C}(F)}\dom(f)$ is finite and coherent then there is an $a=\bigcup\mathcal{A}$ and an $A\in U_a$ by the previous lemma; since $F(A)$ is defined, for each $f\in A$ and $a'\in\mathcal{A}$, $f(a')\subseteq F(A)$, so by the inductive hypothesis, $\{f(a)\}_{a\in\mathcal{A}, f\in\mathcal{C}(F), a\in\dom(f)}$ is coherent.

  Next, suppose $g\in\mathbb{F}^\forall_\phi\setminus \mathcal{C}(F)$.  Since $g\not\subseteq F$, so there is some $a\in\dom(g)$ and some $A\in U_a$ with $F(A)\not\in U_{g(a)}$.  Since $g(a)\not\subseteq F(A)$, by the inductive hypothesis there is a finite $\mathcal{B}\subseteq \mathcal{C}(F(A))$ such that $\mathcal{B}\cup\{g(a)\}$ is not coherent.  By the continuity of $F$, for each $b\in \mathcal{B}$ there is an $a_b\subseteq A$ so that $F(U_{a_b})\subseteq U_b$.  Let $f\subseteq F$ be an element of $\mathbb{F}^\forall_\phi$ with $f(a_b)\supseteq b$ for all $b\in \mathcal{B}$.  Then $\{f,g\}$ is not coherent because $\{a\}\cup\{a_b\}_{b\in \mathcal{B}}\subseteq \mathcal{C}(A)$ is coherent but $\{g(a)\}\cup\{b\}_{b\in \mathcal{B}}$ is not.
\end{proof}


\begin{lemma}\label{thm:fragment}
  For any $A\in\mathbb{S}^\forall_\phi$ and $E\in\mathbb{S}^\exists_\phi$, there exist $a,e$ so that $A\in U_a$, $E\in U_e$, and $(a,e)$ is decisive for $\phi$.
\end{lemma}
\begin{proof}
By induction on $\phi$.  

  Suppose $\phi$ is $\neg\psi$.  By the inductive hypothesis, there are $e_0,a_0$ so that $E\in U_{e_0}$ and $A(E)\in U_{a_0}$ so that $(e_0,a_0)$ is decisive for $\psi$.  Since $E\in A^{-1}(U_{a_0})$, we may choose some $e\supseteq e_0$ so that $E\in U_{e}\subseteq U_{e_0}$ and $A(U_{e})\subseteq U_{a_0}$.  Choose $a$ so that $A\in U_a$ and $a_0\subseteq a(e)$.  Then for any $e'\leq e$, we have $e'\upharpoonright e_0\leq e_0$, so $a(e'\upharpoonright e_0)\leq a(e_0)$, and therefore $(e'\upharpoonright e_0, a(e'\upharpoonright e_0))$ is decisive for $\psi$.  So $(a,e)$ is decisive for $\phi$.
\end{proof}

In light of these lemmas, we can define:
\begin{definition}\label{def:satisfaction}
  We say $\mathfrak{M}\vDash^{\leq A,E}\phi[\vec b]$ if for some (equivalently, every) $(a,e)$ so that $A\in U_a$ and $E\in U_e$ and $(a,e)$ is decisive for $\phi$, $\mathfrak{M}\vDash^{\leq a,e}\phi[\vec b]$.
\end{definition}

\begin{theorem}\label{thm:ctbl_sat}
  If $\mathfrak{M}$ is countably saturated for first-order formulas then $\mathfrak{M}\vDash\phi[\vec b]$ if and only if for every $A\in\mathbb{S}^\forall_\phi$, there is an $E\in\mathbb{S}^\exists_\phi$ such that $\mathfrak{M}\vDash^{\leq A,E}\phi[\vec b]$.
\end{theorem}
\begin{proof}
  By induction on $\phi$.  For atomic formulas, this is the definition.

  Suppose $\phi$ is $\neg\psi$.  If $\mathfrak{M}\vDash\phi[\vec b]$ then $\mathfrak{M}\not\vDash\psi[\vec b]$, so by the inductive hypothesis there is some $A$ so that for every $E$, $\mathfrak{M}\not\vDash^{\leq A,E}\psi[\vec b]$.  Then for any $F\in\mathbb{S}^\forall_\phi$, $\mathfrak{M}\vDash^{\leq F,A}\phi[\vec b]$ since $\mathfrak{M}\not\vDash^{\leq A,F(A)}\psi[\vec b]$.

  Conversely, suppose $\mathfrak{M}\not\vDash\phi[\vec b]$, so by the inductive hypothesis, for every $A$ there is an $E$ such that $\mathfrak{M}\vDash^{\leq A,E}\psi[\vec b]$.  For each $A$, take some $E$ such that $\mathfrak{M}\vDash^{\leq A,E}\psi[\vec b]$, and choose some $a\subseteq A$, $e\subseteq E$ so that $\mathfrak{M}\vDash^{\leq a,e}\psi[\vec b]$.  Let $F_0$ be the set of all pairs of the form $(a,E)$ that we obtain this way.

  We define $F_1\subseteq F_0$ as follows: for each $a$, if there is any $a'\subsetneq a$ and $E'$ with $(a',E')\in F_0$ then $(a,E)\not\in F_1$ for any $E$, and if $a$ is $\subseteq$-minimal such that there is some $E$ with $(a,E)\in F_0$, we choose a single such $E$ and put $(a,E)$ in $F_1$.  Note that we have retained the property that, for each $A$, there is an $a\subseteq A$ and an $E$ with $(a,E)\in F_1$.

  We order $F_1$ as $\{(a_0,E_0),(a_1,E_1),\ldots\}$. We define $F(A)$ by finding the least $i$ such that there is an $a\leq a_i$ with $a\subseteq A$ and setting $F(A)=\max_{j\leq i}E_j$. $F$ is certainly continuous.

  For any $b$ and any $A$ with $b\subseteq F(A)$, there is some $i$ with $F(A)=\max_{j\leq i}E_j$. Choose $a\subseteq A$ large enough so that $a\upharpoonright a_i\leq a_i$ and, for all $j<i$ and all $a^*\leq a_j$, $\{a,a^*\}$ is not coherent. Then for any $a'\leq a$, there is some least $j\leq i$ so that $a'\upharpoonright a_j\leq a_j$ and for all $j'<j$ and $a^*\leq a_{j'}$, $\{a',a^*\}$ is not coherent. Therefore $F(U_{a'})\subseteq U_{\max_{j'\leq j}E_{j'}\upharpoonright b}$.

  Therefore $F\in\mathbb{S}^{\forall}_\phi$ and, for all $A$, $\mathfrak{M}\not\vDash^{\leq F,A}\phi[\vec b]$ since $\mathfrak{M}\vDash^{\leq A,F(A)}\psi[\vec b]$.

  Suppose $\phi$ is $\forall x\psi$.  If $\mathfrak{M}\vDash\phi[\vec b]$ then, by the inductive hypothesis, for every $b'\in|\mathfrak{M}|$ and every $A$, there is an $E$ so that $\mathfrak{M}\vDash^{\leq A,E}\psi[\vec b,x\mapsto b']$.  For every $a\subseteq A$ and every $e$ such that $(a,e)$ is decisive for $\psi$, we have a formula $\hat\psi^{\leq a,e}$.  If there is any $a\subseteq A$ and $e$ so that $\mathfrak{M}\vDash\forall x\hat\psi^{\leq a,e}[\vec b]$ then we may take any $E\in U_e$ and we have $\mathfrak{M}\vDash^{\leq A,E}\phi[\vec b]$.  Suppose not, so for every $a\subseteq A$ and every $e$ so that $(a,e)$ is decisive for $\psi$, there is a $b'\in|\mathfrak{M}|$ so that $\mathfrak{M}\vDash\neg\hat\psi^{\leq a,e}[\vec b,x\mapsto b']$.  There are countably many such pairs, so, by saturation, there is a $b'\in|\mathfrak{M}|$ so that $\mathfrak{M}\vDash\neg\hat\psi^{\leq a,e}[\vec b,x\mapsto b']$ for all $a\subseteq A$ and $e$ so that $(a,e)$ is decisive for $\psi$.  But this contradicts our assumption, because $\mathfrak{M}\vDash\psi[\vec b,x\mapsto b']$, so there is some $E$ so that $\mathfrak{M}\vDash^{\leq A,E}\psi[\vec b,x\mapsto b']$, and therefore some $a\subseteq A$ and $e\subseteq E$ so that $(a,e)$ is decisive and $\mathfrak{M}\vDash^{\leq a,e}\psi[\vec b,x\mapsto b']$.

  If $\mathfrak{M}\not\vDash\phi[\vec b]$ then, by the inductive hypothesis, there is a $b'\in|\mathfrak{M}|$ and an $A$ so that for every $E$, $\mathfrak{M}\not\vDash^{A,E}\psi[\vec b,x\mapsto b']$.  Then for this $A$ and any $E$, we have $\mathfrak{M}\not\vDash^{\leq A,E}\phi[\vec b]$.

  Suppose $\phi$ is $\bigwedge_{i\in I}\psi_i$.  If $\mathfrak{M}\vDash\phi[\vec b]$ then, by the inductive hypothesis, for each $A$ and each $i\in\dom(A)$ there is an $E_i$ so that $\mathfrak{M}\vDash^{\leq A(i),E_i}\psi_i[\vec b]$.  For each $i\in\dom(A)$, let $E'_i=\max_{j\in\dom(A),\ E_j\in\mathbb{S}^\exists_{\psi_i}}E_j$ (so certainly $E_i\leq E'_i$) and let $E=\{E'_i\mid i\in\dom(A)\}$.

  If $\mathfrak{M}\not\vDash\phi[\vec b]$ then, by the inductive hypothesis, there is an $i\in I$ and an $A_i$ so that for any $E_i$, $\mathfrak{M}\not\vDash^{\leq A_i,E_i}\psi_i[\vec b]$.  We may take $A$ with domain $[0,i]$ and $A(i)=A_i$ and we have $\mathfrak{M}\not\vDash^{\leq A,E}\phi[\vec b]$ for any $E$.
\end{proof}

\begin{theorem}\label{thm:main}
  Let $\{\mathfrak{M}_i\}_{i\in\mathbb{N}}$ be a set of structures, let $\phi$ be a formula of $\mathrm{L}_{\omega_1,\omega}$ with free variables $x_1,\ldots,x_n$, and for each $j\leq n$ let $\langle b_j^i\rangle_{i\in\mathbb{N}}$ be a sequence with $b_j^i\in|\mathfrak{M}_i|$ for all $i,j$ and let $b_j\in|\mathfrak{M}^{\mathcal{U}}|$ be the element corresponding to the ultraproduct of the sequence $\langle b_j^i\rangle$.  Then $\mathfrak{M}^{\mathcal{U}}\vDash\phi[b_1,\ldots,b_n]$ if and only if for every $A\in\mathbb{S}^\forall_\phi$ there is an $E\in\mathbb{S}^\exists_\phi$ such that 
\[\{i\mid\mathfrak{M}_i\vDash^{\leq A,E}\phi[b_1^i,\ldots,b_n^i]\}\in\mathcal{U}.\]
\end{theorem}
\begin{proof}
  Suppose $\mathfrak{M}^{\mathcal{U}}\vDash\phi[b_1,\ldots,b_n]$.  Since $\mathfrak{M}^{\mathcal{U}}$ is countably saturated, by Theorem \ref{thm:ctbl_sat}, for each $A\in\mathbb{S}^\forall_\phi$ there is an $E\in\mathbb{S}^\exists_\phi$ such that $\mathfrak{M}^{\mathcal{U}}\vDash^{\leq A,E}\phi[b_1,\ldots,b_n]$.  By Lemma \ref{thm:fragment}, there are $a,e$ so that $A\in U_a$, $E\in U_e$, and $(a,e)$ is decisive for $\phi$, and therefore $\mathfrak{M}^{\mathcal{U}}\vDash^{\leq a,e}\phi[b_1,\ldots,b_n]$. By Lemma \ref{thm:fol} there is a first-order formula $\hat\phi^{\leq a,e}$ so that $\mathfrak{M}^{\mathcal{U}}\vDash\hat\phi^{\leq a,e}[b_1,\ldots,b_n]$, so by the \L{o}\'s Theorem, $\{i\mid \mathfrak{M}_i\vDash\hat\phi^{\leq a,e}[b^i_1,\ldots,b^i_n]\}\in\mathcal{U}$.  Therefore $\{i\mid\mathfrak{M}_i\vDash^{\leq a,e}\phi[b_1,\ldots,b_n]\}\in\mathcal{U}$, and therefore $\{i\mid\mathfrak{M}_i\vDash^{\leq A,E}\phi[b_1,\ldots,b_n]\}\in\mathcal{U}$.

  For the the converse, suppose that for each $A\in\mathbb{S}^\forall_\phi$ there is an $E\in\mathbb{S}^\exists_\phi$ such that
\[S=\{i\mid\mathfrak{M}_i\vDash^{\leq A,E}\phi[b_1^i,\ldots,b_n^i]\}\in\mathcal{U}.\]
Again, by Lemma \ref{thm:fragment}, choose $a,e$ so that $A\in U_a$, $E\in U_e$, and $(a,e)$ is decisive for $\phi$, so that for each $i\in S$, $\mathfrak{M}_i\vDash^{\leq a,e}\phi[b_1^i,\ldots,b_n^i]$.  Then, by Lemma \ref{thm:fol}, for each $i\in S$ $\mathfrak{M}_i\vDash\hat\phi^{\leq a,e}[b_1^i,\ldots,b_n^i]$.  Since $\hat\phi^{\leq a,e}$ is first-order and $S\in\mathcal{U}$, it follows that $\mathfrak{M}^{\mathcal{U}}\vDash\hat\phi^{\leq a,e}[b_1,\ldots,b_n]$, so also $\mathfrak{M}^{\mathcal{U}}\vDash^{\leq A,E}\phi[b_1,\ldots,b_n]$.  Since this holds for every $A$, by Theorem \ref{thm:ctbl_sat} also $\mathfrak{M}^{\mathcal{U}}\vDash \phi[b_1,\ldots,b_n]$.
\end{proof}

\begin{cor}\label{thm:sentence}
  If $\{\mathfrak{M}_i\}_{i\in\mathbb{N}}$ is a sequence of structures, $\mathcal{U}$ is an ultrafilter, and $\sigma$ is a sentence in $\mathrm{L}_{\omega_1,\omega}$ then $\mathfrak{M}^{\mathcal{U}}\vDash\sigma$ if and only if, for every bound $A\in\mathbb{S}^\forall_\sigma$ there is a bound $E\in\mathbb{S}^\exists_\sigma$ such that
\[\{i\mid\mathfrak{M}_i\vDash^{\leq A,E}\sigma\}\in\mathcal{U}.\]
\end{cor}

\begin{cor}\label{thm:cofinite}
  If $\{\mathfrak{M}_i\}_{i\in\mathbb{N}}$ is a sequence of structures and $\sigma$ is a sentence in $\mathrm{L}_{\omega_1,\omega}$ then $\mathfrak{M}^{\mathcal{U}}\vDash\sigma$ for every non-principal ultrafilter $\mathcal{U}$ if and only if, for every bound $A\in\mathbb{S}^\forall_\sigma$ there is a bound $E\in\mathbb{S}^\exists_\sigma$ such that $\{i\mid\mathfrak{M}_i\vDash^{\leq A,E}\sigma\}$ is cofinite.
\end{cor}

\begin{cor}\label{thm:families}
  Let $\mathcal{M}$ be some family of structures and $\sigma$ is a sentence in $\mathrm{L}_{\omega_1,\omega}$.  Then $\mathcal{M}^{\mathcal{U}}\vDash\sigma$ for every $\{\mathfrak{M}_i\}_{i\in\mathbb{N}}\subseteq\mathcal{M}$ and every non-principal ultrafilter $\mathcal{U}$ if and only if, or every bound $A\in\mathbb{S}^\forall_\sigma$ there is a bound $E\in\mathbb{S}^\exists_\sigma$ such that $\mathfrak{M}\vDash^{\leq A,E}\sigma$ for all $\mathfrak{M}\in\mathcal{M}$.
\end{cor}





  

\section{Applications}\label{sec:applications}

\subsection{Interpreting the Definition}

Before giving concrete applications, we consider what our definitions mean for relatively simple formulas.  First, when $\phi$ a first-order formula, both $\mathbb{S}^\forall_\phi$ and $\mathbb{S}^\exists_\phi$ are one point sets---that is, the bound is trivial since there is only one possible ``bound'' on the statement.  In particular, Theorem \ref{thm:main} for first-order formulas says nothing other than \L{o}\'s Theorem.

With a $\Pi^{\mathbb{N}}_1$ formula $\phi=\bigwedge_{i\in\mathbb{N}}\phi_i$, where each $\phi_i$ is first-order, $\mathbb{S}^\forall_\phi$ is isomorphic to $\mathbb{N}$ as an ordered set.  (Formally, $\mathbb{S}^\forall_\phi$ is a function $F$ with domain $[0,n]$ and, for each $i\leq n$, $F(i)$ belongs to a one point set.)  $\mathbb{S}^\exists_\phi$ is a one point set (it is a union of one point sets, and we should treat the elements of these sets as being identical).

When $\phi$ is $\Sigma^{\mathbb{N}}_1$, $\bigvee_{i\in\mathbb{N}}\phi_{i}$m (that is, $\neg\bigwedge_{i\in\mathbb{N}}\neg\phi_i$) where each $\phi_i$ is first-order, $\mathbb{S}^\exists_\phi$ is essentially $\mathbb{N}$ while $\mathbb{S}^\forall_\phi$ is a one point set.

When $\phi$ is a $\Pi^{\mathbb{N}}_2$ formula $\bigwedge_{i\in\mathbb{N}}\bigvee_{j\in\mathbb{N}}\phi_{i,j}$ where each $\phi_{i,j}$ is a first-order set, we get a slightly more interesting case: both $\mathbb{S}^\forall_\phi$ and $\mathbb{S}^\exists_\phi$ are essentially $\mathbb{N}$.  In this case, Theorem \ref{thm:main} is essentially the transfer theorem.

When $\phi$ is a $\Sigma^{\mathbb{N}}_2$ formula $\bigvee_{i\in\mathbb{N}}\bigwedge_{j\in\mathbb{N}}\phi_{i,j}$, $\mathbb{S}^\exists_\phi$ is $\mathbb{N}$, but $\mathbb{S}^\forall_\phi$ now consists of monotone functions from $\mathbb{N}$ to $\mathbb{N}$.  When $\phi$ is $\bigwedge_{i\in\mathbb{N}}\bigvee_{j\in\mathbb{N}}\bigwedge_{k\in\mathbb{N}}\phi_{i,j,k}$, $\mathbb{S}^\exists_\phi$ is still $\mathbb{N}$, but $\mathbb{S}^\forall_\phi$ is now a function with domain $[0,n]$ and range monotone functions from $\mathbb{N}$ to $\mathbb{N}$.  When $F\in\mathbb{S}^\forall_\phi$, we lose nothing by replacing $F$ with the function $F'$ with $\dom(F')=\dom(F)$ and $F'(i)(k)=\max_{i\leq n}F(i)(k)$, so we may assume $F$ is constant---that is, we may view $F$ as an element of $\mathbb{N}$ together with a single monotone function from $\mathbb{N}$ to $\mathbb{N}$.  This is precisely the sort of bound we described in Example \ref{ex:convergence}.

\subsection{Approximate Subgroups}

An example of this sort of interpretation being applied to a $\Pi^{\mathbb{N}}_3$ formula appears in \cite{hrushovski}.

\begin{definition}
  If $X$ and $Y$ are subsets of a group $G$, we say $X$ and $Y$ are \emph{$e$-commensurable} if each is contained in a union of at most $e$ right cosets of the other.
\end{definition}

The main result of \cite{hrushovski} is
\begin{theorem}
  For any function $F:\mathbb{N}^2\rightarrow\mathbb{N}$ and any $k\in\mathbb{N}$, there are $\tilde e,\tilde c$ such that for any group $G$ and any finite $X\subseteq G$ with $|XX^{-1}X|\leq k|X|$, there are $e\leq \tilde e$, $c\leq \tilde c$, and 
\[X_{F(e,c)}\subseteq X_{F(e,c)-1}\subseteq\cdots\subseteq X_1\subseteq X^{-1}XX^{-1}X\]
such that  $X$ and $X_1$ are $e$-commensurable and, for $1\leq m,n<N$, we have
\begin{itemize}
\item $X_n=X_n^{-1}$,
\item $X_{n+1}X_{n+1}\subseteq X_n$,
\item $X_n$ is contained in the union of $c$ right cosets of $X_{n+1}$,
\item $[X_n,X_m]\subseteq X_k$ for any $k<\min\{F(e,c),n+m\}$.
\end{itemize}
\end{theorem}
The function parameter in this theorem is slightly unusual, but its appearance is unsurprising when we realize that the theorem is proven by showing the following in an ultraproduct:
\begin{theorem}
  Let $G^{\mathcal{U}}$ be an ultraproduct of groups.  For any $k$ there are $e, c$ so that, for any internal set $X\subseteq G^{\mathcal{U}}$ with $|XX^{-1}X|\leq k|X|$ and any $N$, there is a sequence 
\[X_N\subseteq X_{N-1}\subseteq\cdots\subseteq X_1\subseteq X^{-1}XX^{-1}X\]
of internal subsets of $G$ such that $X$ and $X_1$ are $e$-commensurable and, for $1\leq m,n<N$, we have
\begin{itemize}
\item $X_n=X_n^{-1}$,
\item $X_{n+1}X_{n+1}\subseteq X_n$,
\item $X_n$ is contained in the union of $c$ right cosets of $X_{n+1}$,
\item $[X_n,X_m]\subseteq X_k$ for any $k<\min\{N,n+m\}$.
\end{itemize}
\end{theorem}
We may view this as an $\mathrm{L}_{\omega_1,\omega}$ sentence over a two-sorted language, with one sort for the group and one sort for the subsets.  In the ultraproduct, this second sort is inhabited by the internal subsets of $G$.  We may write this theorem as a $\Pi^{\mathbb{N}}_3$ sentence, and the finite form above is then exactly what Corollary \ref{thm:families} would predict.

\subsection{Szemer\'edi Regularity}\label{ex:regularity}

Another example comes from the ``strong regularity lemma'' \cite{MR1804820,MR2212136}.

We work in a large finite graph $(X,E)$.
\begin{definition}
  When $U,V\subseteq X$, we define the \emph{edge density between $U$ and $V$} by
\[d(U,V)=\frac{|E\cap(U\times V)|}{|U\times V|}.\]

We say $U,V$ are \emph{$\epsilon$-regular} if whenever $U'\subseteq U$, $V'\subseteq V$, and $\frac{|U'|}{|U|}\geq\epsilon$ and $\frac{|V'|}{|V|}\geq\epsilon$, then
\[\left|d(U,V)-d(U',V')\right|<\epsilon.\]
\end{definition}
Roughly speaking, $U,V$ are $\epsilon$-regular if the edges of $E$ are distributed roughly uniformly between them---any reasonably large subsets have about the ``right'' number of edges.

One version of Szemer\'edi's regularity lemma \cite{szemeredi:MR0369312} says
\begin{theorem}
  For each $\epsilon>0$ there is a $K$ so that, for any finite graph $(X,E)$, there is a partition $X=U_1\cup\cdots\cup U_k$ with $k\leq K$ so that
\[\frac{\sum_{i,j\leq k,U_i\text{ and }U_j\text{ are }\epsilon\text{-regular}}|U_i\times U_j|}{|X|^2}\geq(1-\epsilon).\]
\end{theorem}
Roughly speaking, this says that most points $(x,y)$ belong to a rectangle $U_i\times U_j$ which is $\epsilon$-regular.  There are many variants (for instance, many versions also require that the $U_i$ all have almost exactly the same size, perhaps at the cost of one additional partition piece $U_0$), but these versions can all be derived from each other with a small amount of additional effort.

Many variants of the regularity lemma (such as the ``strong'' regularity lemma \cite{MR1804820}) can be derived from a more general theorem showing that there are always partitions with ``nearly maximal energy''.

\begin{definition}
  If $X=U_1\cup\cdots\cup U_k$ and $X=V_1\cup\cdots\cup V_m$, we say $\{V_j\}_{j\leq m}$ \emph{refines} $\{U_i\}_{i\leq k}$ if, for each $j\leq m$, there is an $i\leq k$ with $V_j\subseteq U_i$.

  The \emph{energy} of a partition $X=U_1\cup\cdots\cup U_k$ is 
\[\mathcal{E}(\{U_i\}_{i\leq k})=\sum_{i,j\leq k}d(U_i,U_j)^2|U_i|\cdot|U_j|.\]
\end{definition}

It is not hard to see that $0\leq\mathcal{E}(\{U_i\}_{i,j\leq k})\leq 1$ and a Cauchy-Schwarz argument shows that when $\{V_j\}_{j\leq m}$ refines $\{U_i\}_{i\leq k}$, $\mathcal{E}(\{U_i\}_{i\leq k})\leq\mathcal{E}(\{V_j\}_{j\leq m})$.

\begin{theorem}\label{thm:tao_regularity}
  For every $\epsilon>0$ and every $F:\mathbb{N}\rightarrow\mathbb{N}$, there is a $K$ so that, for any finite graph $(X,E)$, there is a partition $X=U_1\cup\cdots\cup U_k$ with $k\leq K$ so that whenever $X=V_1\cup\cdots\cup V_m$ is a partition refining $\{U_i\}_{i\leq k}$ with $m\leq F(k)$,
\[\mathcal{E}(\{V_j\}_{j\leq m})<\mathcal{E}(\{U_i\}_{i\leq k})+\epsilon.\]
\end{theorem}

The usual Szemer\'edi's Theorem follows by taking $F(n)=n2^n$ and showing that whenever $(U_i,U_j)$ fails to be regular, we can split $U_i$ and $U_j$ into two sets $U_i=U_i^+\cup U_i^-$ and $U_j=U_j^+\cup U_j^-$ in such a way that the energy increases by an amount on the order of $\epsilon^4|U_i\times U_j|$.  Other variants follow with different choices of $F$.

The presence of the function $F$ suggests that there should be a corresponding fact in an ultraproduct.  We work in a language with a binary relation $\mathbf{E}$ whch will represent the edges of a graph. 



Consider a sequence of finite graphs $G_n=(X_n,E_n)$ with $|X_n|\rightarrow\infty$.  Let $\mathfrak{M}$ be the ultraproduct; we have an edge relation $E$ on $X=|\mathfrak{M}|$ and on each set $X^n$ we have the Loeb measure $\mu^n$ (see \cite{MR1643950} for details). 

In the ultraproduct, the set $E$ is a measurable subset of $X^2$.  We have two $\sigma$-algebras on $X^2$---the $\sigma$-algebra $\mathcal{B}_2$ given directly by the Loeb measure, and the $\sigma$-algebra $\mathcal{B}_1^2$ given by taking the power of the $\sigma$-algebra $\mathcal{B}_1$ on $X$.  The $\sigma$-algebra $\mathcal{B}_1^2$ is generated by rectangles $B\times C$ where $B,C\in\mathcal{B}_1$.

Standard facts about Loeb measure ensure that $\mathcal{B}_1^2\subseteq\mathcal{B}_2$, so we can consider the \emph{projection} of $E$ onto $\mathcal{B}_1^2$: $\mathbb{E}(\chi_E\mid\mathcal{B}_1^2)$ is the function which is measurable with respect to $\mathcal{B}_1^2$ which minimizes the distance $||\chi_E-\mathbb{E}(\chi_E\mid\mathcal{B}_1^2)||_{L^2}$.

The actual existence of the function $\mathbb{E}(\chi_E\mid\mathcal{B}_1^2)$ is an abstract measure-theoretic fact that we cannot express even in $\mathrm{L}_{\omega_1,\omega}$.  But we can express the next best thing: that $\mathbb{E}(\chi_E\mid\mathcal{B}_1^2)$ can be approximated by the projections of $\chi_E$ onto finite sub-algebras.  That is, for each $\epsilon>0$ there is a finite algebra $\mathcal{B}\subseteq\mathcal{B}_1^2$ so that for any other finite algebra $\mathcal{B}'\subseteq\mathcal{B}_1^2$,
\[||\mathbb{E}(\chi_E\mid\mathcal{B}')||_{L^2}<||\mathbb{E}(\chi_E\mid\mathcal{B})||_{L^2}+\epsilon.\]

By standard manipulations in probability theory, we can restrict ourselves to the case where $\mathcal{B}$ has the form $\{U_1,\ldots,U_k\}\times\{U_1,\ldots,U_k\}$, and where we only consider $\mathcal{B}'$ refining $\mathcal{B}$ (by replacing $\mathcal{B}'$ with the common refinement of $\mathcal{B}$ and $\mathcal{B}'$).  Finally, observing that $||\mathbb{E}(\chi_E\mid\mathcal{B})||_{L^2}^2$ is precisely $\mathcal{E}(\{U_i\}_{i\leq k})$, we see that we are considering the statement:
\begin{quote}
  For each $\epsilon>0$ there is a partition $X=U_1\cup\cdots\cup U_k$ such that whenever $X=V_1\cup\cdots\cup V_m$ refines $\{U_i\}_{i\leq k}$,
\[\mathcal{E}(\{V_j\}_{j\leq m})^2<\mathcal{E}(\{U_i\}_{i\leq k})^2+\epsilon.\]
\end{quote}
Technically we need to add some predicates to our language making it possible to write down formulas about the measure; see \cite{goldbring:_approx_logic_measure} for one way to do this.  We also need to add a sort for internal subsets of $X$.  After these tweaks to the language, this is a $\Pi^{\mathbb{N}}_3$ sentence.  Stated like this, we see that the strong regularity lemma is precisely what Corollary \ref{thm:families} predicts.

\subsection{Hypergraph Regularity}

Building on the work in the previous subsection, consider an ultraproduct of finite $d$-ary hypergraphs: $\mathfrak{M}=(X,H)$ with $H\subseteq X^d$, where the Loeb measure gives us measures $\mu^d$ on each $X^d$, and assume $\mu^d(H)>0$.

The Loeb measure gives a $\sigma$-algebra $\mathcal{B}_d$ on $X^d$ containing all internal subsets.  But for any $r<d$ we can define a $\sigma$-algebra $\mathcal{B}_{d,r}$ consisting of sets of the form
\[\{(x_1,\ldots,x_d)\mid\forall s\in {d\choose r} \vec x_s\in B_s\}\]
where each $B_s\in\mathcal{D}_r$.  Then $\mathcal{B}_1^d=\mathcal{B}_{d,1}\subseteq\mathcal{B}_{d,2}\subseteq\cdots\subseteq\mathcal{B}_{d,d}=\mathcal{B}_d$, and in general these inclusions are strict.

We can describe a series of projections of $\chi_H$.  For notational simplicity, consider the case where $d=3$:
\begin{theorem}
  For any $\epsilon>0$ there is a finite $\mathcal{B}\subseteq\mathcal{B}_{3,2}$ so that
\begin{itemize}
\item for any $\mathcal{B}'\subseteq\mathcal{B}_{3,2}$,
\[||\mathbb{E}(\chi_H\mid\mathcal{B}')||_{L^2}<||\mathbb{E}(\chi_H\mid\mathcal{B})||_{L^2}+\epsilon,\]
and
\item for $\epsilon'>0$ there is a $\mathcal{C}\subseteq\mathcal{B}_{2,1}$ so that, for each $B\in\mathcal{B}$ and any $\mathcal{C}'\subseteq\mathcal{B}_{2,1}$,
\[||\mathbb{E}(\chi_C\mid\mathcal{C}')||_{L^2}<||\mathbb{E}(\chi_C\mid\mathcal{C})||_{L^2}+\epsilon'.\]
\end{itemize}
\end{theorem}
After some work combining conjunctions to get this statement in a prenex form, we get a $\Pi^{\mathbb{N}}_5$ sentence.  Theorem \ref{thm:families} gives a corresponding theorem about finite partitions in finite graphs, though it will clearly be quite complicated.  When $d>3$, the statement becomes even more complicated: the analogous statement for $d$-ary hypergraphs will be $\Pi^{\mathbb{N}}_{2d-1}$.

A special case (analogous to the way Szemer\'edi regularity follows from a special case of strong regularity) is known as hypergraph regularity \cite{MR2373376,rodl:MR2069663}; the statement is quite complicated, which is not surprising given that it corresponds to a fairly complicated $\mathrm{L}_{\omega_1,\omega}$ formula.

\subsection{Hilbertianity}

An example of how these results apply to non-prenex formulas is given by the Gilmore-Robinson characterization of Hilbertian fields.

\begin{definition}
  When $K$ is a field, $\vec T$ is a set of variables, $f_1(\vec T,X),\ldots,f_m(\vec T,X)$ are polynomials with coefficients in $K(\vec T)$ which are irreducible in $K(\vec T)[X]$, and $g$ is a polynomial in $K[\vec T]$, $H_K(f_1,\ldots,f_m;g)$ is the set of $\vec a\in K$ suc that $g(\vec a)\neq 0$ and each $f_i(\vec a,X)$ is defined and irreducible in $K[X]$.

A field is \emph{Hilbertian} if every set $H_K(f_1,\ldots,f_m;g)$ is nonempty.
\end{definition}
Hilbertianity is incompatible with algebraic closure except in trivial cases: roughly speaking, in a Hilbertian field, polynomials which always factor must factor uniformly.  That is, if, for every $a\in K$, the polynomial $f(a,X)\in K[X]$ factors, then actually $f(Y,X)\in K(Y)[X]$ must factor as well.

The Gilmore-Robinson characterization of Hilbertianity, restricted (for convenience) to characteristic $0$, says:
\begin{theorem}[\cite{MR0072088}]
  When $K$ has characteristic $0$, $K$ is Hilbertian if and only if there is an $t\in K^{\mathcal{U}}\setminus K$ such that $\overline{K(t)}\cap K^\mathcal{U}=K(t)$.
\end{theorem}
When $K$ is countable, this becomes a formula of $\mathrm{L}_{\omega_1,\omega}$.  We take the language to be the language of $K$-rings---that is, the language of rings together with a constant symbol for each element of $k$.  Then the statement that $t\in K^{\mathcal{U}}\setminus K$ becomes 
\[\exists t\bigwedge_{k\in K}(t\neq k).\]

There are only countably many elements of $K(t)$, so we can quantify over them using a $\bigwedge$ quantifier: the statement $\overline{K(t)}\cap K^\mathcal{U}=K(t)$ becomes
\[\bigwedge_{p\in K(y)[x]}\forall x\, p(t,x)=0\rightarrow \bigvee_{u\in K(y)}x=u(t).\]
Putting these together, we get the formula
\[\exists t\left[\bigwedge_{k\in K}(t\neq k)\wedge \bigwedge_{p\in K(y)[x]}\forall x\left(p(t,x)=0\rightarrow \bigvee_{u\in K(y)}x=u(t)\right)\right].\]
This behaves roughly like a $\Sigma^{\mathbb{N}}_3$ sentence.

Let $\mathcal{S}$ consist of finite subsets of $K$.  By Corollary \ref{thm:sentence} (with the sequence $\mathfrak{M}_i=K$ for all $i$) together with a little work interpreting the result, we get:
\begin{quote}
  For every function $Q:((\mathcal{S}\times\mathbb{N})\rightarrow\mathbb{N})\rightarrow(\mathcal{S}\times\mathbb{N})$ which is continuous (where $\mathcal{S}\times\mathbb{N}$ and $\mathbb{N}$ have the discrete topology and the topology on $(\mathcal{S}\times\mathbb{N})\rightarrow\mathbb{N}$ is generated by sub-basic sets of the form $\{F\mid F(S,n)=m\}$), there is an $F:(\mathcal{S}\times\mathbb{N})\rightarrow\mathbb{N}$ such that, taking $Q(F)=(S,n)$, there is a $t\in K\setminus S$ so that for any $p\in K[t^{-1},t,x]$ of degree at most $n$, if there is an $x\in K$ with $p(x)=0$ then there is a $u\in K[t^{-1},t]$ of degree at most $F(S,n)$ with $p(u)=0$.
\end{quote}

\subsection{Other Examples}

The technique of this paper has been used in several places to obtain constructive or explicit information from proofs which use ultraproducts.  As the other examples show, the case of $\Pi^{\mathbb{N}}_3$ sentences has been investigated independently several times.  Therefore these techniques are mostly useful in proofs whose intermediate steps involve sentences of higher complexity.

In \cite{MR3643744}, the author applied this technique to a statement involving showing that two double limits are equal---that
\[\lim_m\lim_k a_{m,k}=\lim_k\lim_m a_{m,k}\]
for a particular sequence $a_{m,k}$.  At its fullest complexity, this is
\begin{quote}
  For every $\epsilon>0$ there are $m$ and $k$ so that for all $m'>m$, $k'>k$ there are $l$ and $n$ so that for all $l'>l$ and $n'>n$, $|a_{m',l'}-a_{n',k'}|<\epsilon$
\end{quote}
which is $\Pi^{\mathbb{N}}_5$.  In fact, \cite{MR3643744} analyzes the slightly less complicated $\Pi^{\mathbb{N}}_4$ statement 
\begin{quote}
  For every $\epsilon>0$, $m$, and $k$, there are $m'>m$ and $k'>k$ so that for every $l$ and $n$ there are $l'>l$ and $n'>$ so that $|a_{m',l'}-a_{n',k'}|<\epsilon$
\end{quote}
which has one fewer connective to deal with and suffices for the intended application \cite{towsner:banach}.  (Nonetheless, the proof goes through sentences of higher complexity.)

In \cite{1609.07509}, the Simmons and the author applied this technique to a result in differential algebra \cite{HTKM}.  The main result of \cite{HTKM} is $\Pi^{\mathbb{N}}_2$, so the final bounds are of the form ``for every $n$ and $d$ in $\mathbb{N}$ there is a $b$ in $\mathbb{N}$'', but various intermediate steps have higher complexity.

In \cite{MR3846327}, Goldbring, Hart, and the author apply a variant of these methods to sentences in continuous logic: in \cite{MR3694564}, Boutonnet, Chifan, and Ioana show that certain $\Pi_1$ factors are not elmentarily equivalent by showing that their ultraprowers are non-isomorphic, but without constructing actual sentences.  Goldbring and Hart \cite{MR3704741} analyzed the complexity of the sentences distinguishing these $\Pi_1$ factors, but were not able to identify the precise sentences.  The analysis in \cite{1701.07928} proceeds by noting that the properties distinguishing the ultrapowers are expressible in $\mathrm{L}_{\omega_1,\omega}$ (more precisely, some continuous variant of it), and then using techniques inspired by this paper to translate those into continuous sentences distinguishing the factors.

The translation given here is \emph{not} designed for continuous logic, and the sentences it produced required a substantial amount of tweaking by hand to get correct arguments.
\begin{question}
  What is the correct semantics to give a version of Theorem \ref{thm:ctbl_sat} for continuous logic?
\end{question}

\printbibliography
\end{document}